\documentclass[12pt,a4paper]{amsart}
\usepackage{a4wide}
\usepackage{amsmath,amssymb,amsfonts,amsthm}
\usepackage{enumerate}
\usepackage{amscd}
\usepackage{graphicx}
\usepackage[all]{xy}
\usepackage{hhline}

\newtheorem{theorem}{Theorem}[section]
\newtheorem{lemma}[theorem]{Lemma}
\newtheorem{corollary}[theorem]{Corollary}
\newtheorem{proposition}[theorem]{Proposition}

 \theoremstyle{definition}
 \newtheorem{definition}[theorem]{Definition}
 \newtheorem{remark}[theorem]{Remark}
\newtheorem{remarks}[theorem]{Remarks}
 \newtheorem{example}[theorem]{Example}
\newtheorem{examples}[theorem]{Examples}

\numberwithin{equation}{section}

\newcommand {\N}{\mathbb{N}} 
\newcommand {\Z}{\mathbb{Z}} 
\newcommand {\R}{\mathbb{R}} 
\newcommand {\Q}{\mathbb{Q}} 
\newcommand {\C}{\mathbb{C}} 

\newcommand {\sph}{\mathbb{S}} 

\newcommand{\BB}{\mathcal{B}}
\newcommand{\CC}{\mathcal{C}}
\newcommand{\DD}{\mathcal{D}}

\newcommand{\FF}{\mathcal{F}}

\newcommand{\UU}{\mathcal{U}}

\newcommand{\tore}{\mathbb{T}}

\DeclareMathOperator{\Fix}{Fix}
\DeclareMathOperator{\CA}{CA}

\DeclareMathOperator{\Id}{Id}

 \newcommand{\Set}{\mathrm{Set}}
\newcommand{\Grp}{\mathrm{Grp}}
\newcommand{\Rng}{\mathrm{Rng}}

\newcommand{\Fld}{\mathrm{Fld}}
\newcommand{\Mod}{\mathrm{Mod}}

\newcommand{\Top}{\mathrm{Top}}
\newcommand{\hTop}{\mathrm{hTop}}
\newcommand{\Vect}{\mathrm{Vec}}

\newcommand{\Man}{\mathrm{Man}}
\newcommand{\CHT}{\mathrm{CHT}}
\newcommand{\Aal}{\mathrm{Aal}}

\newcommand{\Setf}{\mathrm{Set}^{\scriptstyle{f}}}
\newcommand{\Modfg}{\mathrm{Mod}^{\scriptstyle{f-g}}}
\newcommand{\ModArt}{\mathrm{Mod}^{\scriptstyle{Art}}}

\newcommand{\Vectfd}{\mathrm{Vec}^{\scriptstyle{f-d}}}

\begin{document}
 \title[Surjunctivity and reversibility of cellular automata]{Surjunctivity and reversibility of cellular automata over concrete categories}
 \author{Tullio Ceccherini-Silberstein}
\address{Dipartimento di Ingegneria, Universit\`a del Sannio, C.so
Garibaldi 107, 82100 Benevento, Italy}
\email{tceccher@mat.uniroma3.it}
\author{Michel Coornaert}
\address{Institut de Recherche Math\'ematique Avanc\'ee,
UMR 7501,                                             Universit\'e  de Strasbourg et CNRS,
                                                 7 rue Ren\'e-Descartes,
                                               67000 Strasbourg, France}
\email{coornaert@math.unistra.fr}
\date{\today}
\begin{abstract}
Following ideas developed by Misha Gromov,
we investigate surjunctivity and reversibility properties of cellular automata defined over certain concrete categories.
 \end{abstract}
\maketitle
\begin{center}
\emph{Dedicated to Alessandro Fig\`a Talamanca}
\end{center}

\section{Introduction}

A cellular automaton is an abstract machine which takes as input a configuration of a universe and produces as output another configuration. The universe consists of cells and a configuration is described by the state of each cell of the universe.
There is an input   and an output state set and these two sets  may be distinct. 
 The state sets are also called the sets of colors, the sets of symbols,  or the alphabets. 
The transition rule of a cellular automaton must obey two important properties, namely    locality and time-independence. Locality means that the output state of a given cell only depends on the input states of a finite number of its neighboring cells possibly including the cell itself.  
\par
In the present paper, we restrict ourselves to cellular automata for which the universe is a group.
The cells are the elements of the group. 
If  the input alphabet is denoted by $A$, the output alphabet by $B$,   and the group by $G$,  this means that a cellular automaton is given by  a map $\tau \colon A^G \to B^G$,
 where  $A^G := \{x \colon G \to A \}$ is the  set of all input configurations and $B^G := \{y \colon G \to B \}$   is the set of all possible output configurations.
 Besides locality, we will also always require a  symmetry condition for $\tau$, namely that it commutes with the  shift actions of $G$ on $A^G$ and $B^G$  (see Section \ref{s:ca} for precise definitions).
In the case when $G = \Z^2$ and $A = B$ is a finite set, such cellular automata were first considered by John von Neumann in the late 1940s to serve as theoretical models for   
self-reproducing robots.
Subsequently, cellular automata over $\Z$, $\Z^2$ or  $\Z^3$ were widely used to model complex systems arising in natural or physical sciences.
On the other hand, the mathematical study of cellular automata developed as a flourishing branch of theoretical computer science with numerous connections to abstract algebra, topological dynamical systems, ergodic theory, statistics and probability theory.
\par
One of the most famous results in the theory of cellular automata is the Garden of Eden theorem established by Moore \cite{moore} and Myhill \cite{myhill} in the early 1960s.
It asserts that a cellular automaton $\tau \colon A^G \to A^G$, with $G = \Z^d$ and $A$ finite, is surjective if and only if it is pre-injective
(here \emph{pre-injective} means that two configurations having  the same image by $\tau$ must be equal if they coincide outside a finite number of cells).
The name of this theorem comes from the fact that a configuration that is not in the image of a cellular automaton $\tau$ is sometimes called a \emph{Garden of Eden} for $\tau$ because in a dynamical picture of the universe, obtained by iterating $\tau$, such a configuration  can only appear at time $0$. Thus, the surjectivity of $\tau$ is equivalent to the absence of Garden of Eden configurations.  
At the end of the last century, the Garden of Eden theorem was extended to any amenable group $G$ in \cite{ceccherini}
and it is now known \cite{bartholdi} that the class of amenable groups is precisely the largest class of groups for which the Garden of Eden theorem is valid. 
Observe that an immediate consequence of the  Garden of Eden theorem is that  every injective cellular automaton  
$\tau \colon A^G \to A^G$, with $G$ amenable and $A$ finite, is surjective and hence bijective.
The fact that injectivity implies surjectivity, a property known as \emph{surjunctivity} 
\cite{gottschalk}, was extended by 
Gromov \cite{gromov-esav} and Weiss \cite{weiss-sgds}  to all cellular automata $\tau \colon A^G \to A^G$ with finite alphabet over sofic groups.
The class of sofic groups is a class of groups introduced by Gromov \cite{gromov-esav}  containing in particular all amenable groups and all residually finite groups.
Actually,  there is no known example of a group that is not sofic. 
\par
Let us note that in the classical literature on cellular automata, the alphabet sets are most often assumed to be finite. With these hypotheses, topological methods based on properties of  compact spaces  may be used   since $A^G$ is compact for the prodiscrete topology when $A$ is finite (see Subsection \ref{ss:config-shift}).
For example, one easily deduces from compactness that every bijective cellular automaton 
$\tau \colon A^G \to B^G$ is reversible when $A$ is finite (a cellular automaton is called \emph{reversible} if it is bijective and its inverse map is also a cellular automaton, see Subsection \ref{ss:reversible}). On the other hand, in the infinite alphabet case, there exist bijective cellular automata that are not reversible.
\par 
The aim of the present paper is to investigate cellular automata $\tau \colon A^G \to B^G$ when the alphabets $A$ and $B$ are (possibly infinite) objects in some concrete category $\CC$ and the local defining rules of $\tau$ are $\CC$-morphisms (see Section \ref{s:ca-concrete-categories} for precise definitions).
For example, $\CC$ can be the category of (let us say left) modules over some ring $R$, the category of topological spaces, or some of their subcategories.
When $\CC$ is the category of vector spaces over a field $K$, 
or, more generally, the category of left modules over a ring $R$,
we recover the notion of a \emph{linear cellular automaton} studied in 
\cite{garden}, \cite{israel}, \cite{modsofic}, \cite{artinian}.
When $\CC$ is the category of affine algebraic sets over a field $K$,
this gives the notion of an \emph{algebraic cellular automaton} as in \cite{algebraic}.
\par
Following ideas developed by Gromov \cite{gromov-esav}, we shall give sufficient conditions for a concrete category $\CC$ that guarantee surjunctivity of all $\CC$-cellular automata $\tau \colon A^G \to A^G$ when the group $G$ is residually finite 
(see Section \ref{sec:surjunctivity-rf-groups}).
We shall also describe conditions on $\CC$ implying that all bijective $\CC$-automata are reversible (see Section \ref{sec:reversibility}).
\par
The present paper is mostly expository and collects results established in Gromov's seminal article  \cite{gromov-esav} and  in a series of papers written by the authors 
\cite{garden}, \cite{israel}, \cite{modsofic}, \cite{artinian}, \cite{induction},
\cite{periodic}, \cite{algebraic}. 
However, our survey contains some new proofs as well as some new results. 
On the other hand, we hope our concrete categorical approach will   help the reader to a better understanding of this fascinating subject connected to so many contemporary branches of mathematics and theoretical computer science.

 \section{cellular automata}
 \label{s:ca}

In this section, we have gathered background material on cellular automata over groups.
The reader is referred to our monograph \cite{book} for a more detailed exposition.

\subsection{The space of configurations and the shift action}
\label{ss:config-shift}
Let $G$ be a group and let $A$ be a set (called the \emph{alphabet} or the set of \emph{colors}).
The set $$A^G = \{x \colon G \to A \}$$
is endowed with its \emph{prodiscrete topology}, i.e., the product topology obtained by taking the discrete topology on each factor $A$ of $A^G = \prod_{g \in G} A$.
Thus, if $x \in A^G$, a base of open neighborhoods of $x$ is provided by the sets
$$V(x,\Omega) := \{y \in A^G : x\vert_\Omega = y\vert_\Omega \},$$
where $\Omega$ runs over all finite subsets of $G$ and $x\vert_\Omega \in A^\Omega$ denotes the restriction of  $x \in A^G$ to $\Omega$. 

The space $A^G$, which is called the space of \emph{configurations}, is Hausdorff and totally disconnected. It it is compact if and only if the alphabet $A$ is finite.

\begin{example}
If $G$ is countably infinite and $A$ is finite of cardinality $\vert A \vert \geq 2$, then $A^G$ is homeomorphic to the Cantor set.
This is the case for $G = \Z$ and $A = \{0,1\}$, where $A^G$ is the space of bi-infinite sequences of $0$'s and $1$'s.
\end{example}

There is a natural continuous left action of $G$ on $A^G$ given by
\begin{align*}
G \times A^G  &\to A^G \\
(g,x) &\mapsto gx
\end{align*}
where
$$
gx(h) = x(g^{-1}h) \quad \forall h \in G.
$$
This action is called the $G$-\emph{shift} on $A^G$. 


The study of the shift action on $A^G$ is the central theme in \emph{symbolic dynamics}.

\subsection{Cellular automata}
\label{ss:ca}
\begin{definition}
Let $G$ be a group. 
Given two alphabets $A$ and $B$, a map $\tau \colon A^G \to B^G$ is called a \emph{cellular automaton} if 
 there exist a finite subset $M \subset G$ and a map $\mu_M \colon A^M \to B$ such that
\begin{equation}
\label{e;local-property}
(\tau(x))(g) = \mu_M((g^{-1}x)\vert_M) \quad \forall x \in A^G, \forall g \in G,
\end{equation} 
where $(g^{-1}x)\vert_M$ denotes the restriction of the configuration $g^{-1}x$ to $M$.
Such a set $M$ is called a \emph{memory set} and the map $\mu_M \colon A^M \to B$ is called the associated  \emph{local defining map}.
\end{definition}

\begin{example}
\label{ex:identity-is-ca}
The identity map $\Id_{A^G} \colon A^G \to A^G$ is a cellular automaton with memory set $M = \{1_G\}$ and local defining map the identity map 
$\mu_M = \Id_{A} \colon A^{M} = A \to A$.
\end{example}

\begin{example}
If we fix an element $b_0 \in B$, then the constant map $\tau \colon A^G \to B^G$, defined by 
$\tau(x)(g) = b_0$ for all $x \in A^G$ and $g \in G$, is a cellular automaton with memory set $M = \varnothing$.
\end{example}

\begin{example}
If we fix an element $g_0 \in G$, then the map $\tau \colon A^G \to A^G$, defined by 
$\tau(x)(g) = x(g g_0)$ for all $x \in A^G$ and $g \in G$,
is a cellular automaton with memory set $M = \{g_0\}$. 
\end{example}

 \begin{example}[The majority action on $\Z$]
Let $G = \Z$, $A = \{0,1\}$, $M = \{-1,0,1\}$, and $\mu_M:A^M \equiv A^3 \to A$
defined by $\mu_M(a_{-1},a_0,a_1) = 1$ if $a_{-1}+a_0+a_1 \geq 2$ and $\mu_M(a_{-1},a_0,a_1)=0$ otherwise. 

\begin{figure}[h!]
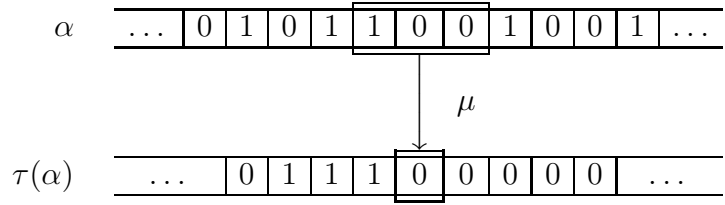

$$
\begin{array}{rcccccccccccccc}
\hhline{~~-----|===|-----}
\alpha&&\dots & \multicolumn{1}{|c}{0} & \multicolumn{1}{|c}{1} & \multicolumn{1}{|c}{0} & \multicolumn{1}{|c|}{1} & \multicolumn{1}{c}{1} & \multicolumn{1}{|c|}{0} & \multicolumn{1}{c|}{0} & \multicolumn{1}{c|}{1} & \multicolumn{1}{c|}{0} & \multicolumn{1}{c|}{0} & \multicolumn{1}{c|}{1} & \dots\\
\cline{3-7}\cline{11-15}
\hhline{~~~~~~~|===|~~~~~}
&&&&&&&&  {\Bigg \downarrow} &\mu&&&&&\\
\hhline{~~------|=|------}
\tau(\alpha) & &\multicolumn{2}{c}{\dots} & \multicolumn{1}{|c}{0} & \multicolumn{1}{|c}{1} & \multicolumn{1}{|c}{1} & \multicolumn{1}{|c|}{1} & \multicolumn{1}{c|}{0} & \multicolumn{1}{c|}{0} & \multicolumn{1}{c|}{0} & \multicolumn{1}{c|}{0} & \multicolumn{1}{c|}{0} & \multicolumn{2}{c}{\dots}\\
\cline{3-8}\cline{10-15}
\hhline{~~~~~~~~|=|~~~~~~}
\end{array}
$$
\caption{The cellular automaton defined by the majority action on $\Z$.}\label{fig:maj}
\end{figure}

\end{example}

\begin{example}[Hedlund's marker \cite{hedlund}]
Let $G = \Z$, $A = \{0,1\}$, $M = \{-1,0,1,2\}$, and $\mu_M \colon A^M \equiv A^4 \to A$
defined by $\mu_M(a_{-1},a_0,a_1,a_2) = 1 - a_0$ if $(a_{-1},a_1,a_2) = (0,1,0)$ and $\mu_M(a_{-1},a_0,a_1,a_2) = a_0$ otherwise. The corresponding cellular automaton $\tau \colon A^\Z \to A^\Z$ is a nontrivial involution of $A^G$.

\begin{figure}[h!]
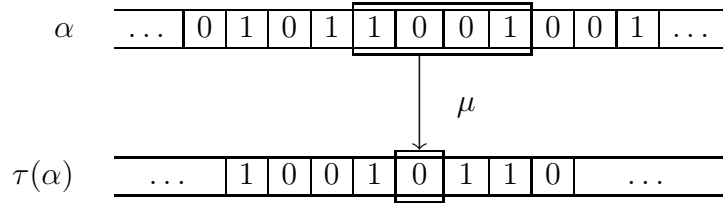

$$
\begin{array}{rcccccccccccccc}
\hhline{~~-----|====|----}
\alpha&&\dots & \multicolumn{1}{|c}{0} & \multicolumn{1}{|c}{1} & \multicolumn{1}{|c}{0} &
\multicolumn{1}{|c|}{1} & \multicolumn{1}{c}{1} & \multicolumn{1}{|c|}{0} & \multicolumn{1}{c|}{0} &
\multicolumn{1}{c|}{1} & \multicolumn{1}{c|}{0} & \multicolumn{1}{c|}{0} & \multicolumn{1}{c|}{1} &
\dots\\
\cline{3-7}\cline{11-15}
\hhline{~~~~~~~|====|~~~~}
&&&&&&&&  {\Bigg \downarrow} &\mu&&&&&\\
\hhline{~~------|=|------}
\tau(\alpha) & &\multicolumn{2}{c}{\dots} & \multicolumn{1}{|c}{1} & \multicolumn{1}{|c}{0} &
\multicolumn{1}{|c}{0} & \multicolumn{1}{|c|}{1} & \multicolumn{1}{c|}{0} & \multicolumn{1}{c|}{1} &
\multicolumn{1}{c|}{1} & \multicolumn{1}{c|}{0} & \multicolumn{3}{c}{\dots}\\
\cline{3-8}\cline{10-15}
\hhline{~~~~~~~~|=|~~~~~~}
\end{array}
$$
\caption{The cellular automaton defined by the Hedlund marker.}\label{fig:hed}
\end{figure}
\end{example}

\begin{example}[Conway's Game of Life]
Life was introduced by J. H. Conway in the 1970's and popularized by M. Gardner. From a theoretical computer science point of view, it is important because it has the power of a universal Turing machine, that is, anything that can be computed algorithmically can be computed by using the Game of Life.

Let $G = \Z^2$ and $A = \{0,1\}$. 
\emph{Life} is the cellular automaton 
$$
\tau \colon \{0,1\}^{\Z^2} \to \{0,1\}^{\Z^2}
$$
with memory set $M = \{-1,0,1\}^2 \subset \Z^2$ and local defining map $\mu \colon A^M \to A$ given by
\begin{equation}
\label{GoL}
\mu(y) =\left\{
\begin{array}{ll}
1 &\mbox{if}\left\{ \begin{array}{l} \mbox{ } \displaystyle\sum_{m
\in M} y(m) = 3 \\
\mbox{or } \displaystyle\sum_{m \in M} y(m) = 4 \mbox{ and } y((0,0)) = 1
\end{array} \right. \\
0 &\mbox{otherwise}
\end{array} \right.
\end{equation}
for all $y \in A^M$.
One thinks of an element $g$ of $G = \Z^2$ as a
``cell'' and the set $gM$ (we use multiplicative notation) as the set consisting of
its eight neighboring cells, namely the North, North-East, East, South-East, South,
South-West, West and North-West cells. We interpret state $0$ as corresponding  to the \emph{absence} of life while state $1$ corresponds to the  \emph{presence} of life.  We  thus refer to  cells in state $0$ as \emph{dead} cells and to cells  in state $1$ as  \emph{live} cells. Finally, if $x \in A^{\Z^2}$ is a configuration at time $t$, then $\tau(x)$ represents the evolution of the configuration  at time $t+1$.
Then the cellular automaton in \eqref{GoL} evolves as follows.
\begin{itemize}
\item {\it Birth:} a cell that is dead at time $t$ becomes alive at time $t+1$ if and only if  three of its neighbors are alive at time $t$.
\item{\it Survival:}  a cell that is alive at time $t$ will remain
alive at time $t+1$ if and only if it has  exactly two or three
live neighbors at time $t$.
\item{\it Death by loneliness:} a live cell that has at most one live
neighbor at time $t$ will be dead at time $t+1$.
\item{\it Death by overcrowding:} a cell that is alive at time $t$
and has four or more live neighbors at time $t$, will be
dead at time $t+1$.
\end{itemize}
Figure~\ref{fig:gol} illustrates all these cases.

\begin{figure}[h!]
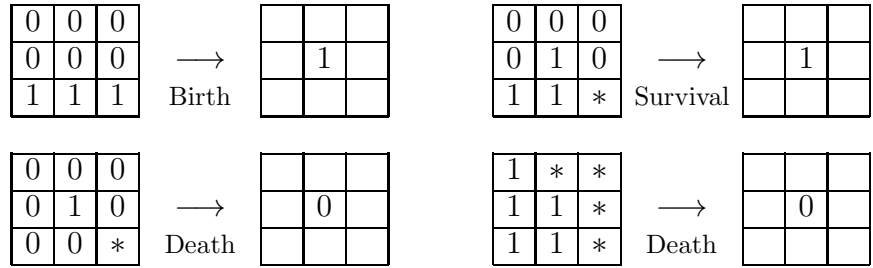

$$
\begin{array}{ccccccccccccccc}
\cline{1-3}\cline{5-7}\cline{9-11}\cline{13-15}
\multicolumn{1}{|c}{0} & \multicolumn{1}{|c}{0} & \multicolumn{1}{|c|}{0} & \phantom{\mbox{\footnotesize Survival}}&\multicolumn{1}{|c}{\phantom{1}} & \multicolumn{1}{|c}{} & \multicolumn{1}{|c|}{\phantom{1}}
&\phantom{lungo}&
\multicolumn{1}{|c}{0} & \multicolumn{1}{|c}{0} & \multicolumn{1}{|c|}{0} & &\multicolumn{1}{|c}{} & \multicolumn{1}{|c}{} & \multicolumn{1}{|c|}{}
\\
\cline{1-3}\cline{5-7}\cline{9-11}\cline{13-15}
\multicolumn{1}{|c}{0} & \multicolumn{1}{|c}{0} & \multicolumn{1}{|c|}{0} & \longrightarrow &\multicolumn{1}{|c}{} & \multicolumn{1}{|c}{1} & \multicolumn{1}{|c|}{}
&&
\multicolumn{1}{|c}{0} & \multicolumn{1}{|c}{1} & \multicolumn{1}{|c|}{0} & \longrightarrow &\multicolumn{1}{|c}{} & \multicolumn{1}{|c}{1} & \multicolumn{1}{|c|}{}
\\
\cline{1-3}\cline{5-7}\cline{9-11}\cline{13-15}
\multicolumn{1}{|c}{1} & \multicolumn{1}{|c}{1} & \multicolumn{1}{|c|}{1} & \mbox{\footnotesize Birth} &\multicolumn{1}{|c}{} & \multicolumn{1}{|c}{} & \multicolumn{1}{|c|}{}
&&
\multicolumn{1}{|c}{1} & \multicolumn{1}{|c}{1} & \multicolumn{1}{|c|}{*} & \mbox{\footnotesize Survival} &\multicolumn{1}{|c}{} & \multicolumn{1}{|c}{} & \multicolumn{1}{|c|}{}
\\
\cline{1-3}\cline{5-7}\cline{9-11}\cline{13-15}
\\
\cline{1-3}\cline{5-7}\cline{9-11}\cline{13-15}
\multicolumn{1}{|c}{0} & \multicolumn{1}{|c}{0} & \multicolumn{1}{|c|}{0} & &\multicolumn{1}{|c}{\phantom{1}} & \multicolumn{1}{|c}{} & \multicolumn{1}{|c|}{\phantom{1}}
&&
\multicolumn{1}{|c}{1} & \multicolumn{1}{|c}{*} & \multicolumn{1}{|c|}{*} & &\multicolumn{1}{|c}{\phantom{1}} & \multicolumn{1}{|c}{} & \multicolumn{1}{|c|}{\phantom{1}}
\\
\cline{1-3}\cline{5-7}\cline{9-11}\cline{13-15}
\multicolumn{1}{|c}{0} & \multicolumn{1}{|c}{1} & \multicolumn{1}{|c|}{0} & \longrightarrow &\multicolumn{1}{|c}{} & \multicolumn{1}{|c}{0} & \multicolumn{1}{|c|}{}
&&
\multicolumn{1}{|c}{1} & \multicolumn{1}{|c}{1} & \multicolumn{1}{|c|}{*} & \longrightarrow &\multicolumn{1}{|c}{} & \multicolumn{1}{|c}{0} & \multicolumn{1}{|c|}{}
\\
\cline{1-3}\cline{5-7}\cline{9-11}\cline{13-15}
\multicolumn{1}{|c}{0} & \multicolumn{1}{|c}{0} & \multicolumn{1}{|c|}{*} & \mbox{\footnotesize Death} &\multicolumn{1}{|c}{} & \multicolumn{1}{|c}{} & \multicolumn{1}{|c|}{}
&&
\multicolumn{1}{|c}{1} & \multicolumn{1}{|c}{1} & \multicolumn{1}{|c|}{*} & \mbox{\footnotesize Death} &\multicolumn{1}{|c}{} & \multicolumn{1}{|c}{} & \multicolumn{1}{|c|}{}
\\
\cline{1-3}\cline{5-7}\cline{9-11}\cline{13-15}
\end{array}
$$
\caption{The evolution of a cell in the Game of Life. The symbol $*$ represents any symbol in \{0,1\}.}\label{fig:gol}
\end{figure}

\end{example}

 Observe that if $\tau \colon A^G \to B^G$ is a cellular automaton with memory set $M$ and local defining map $\mu_M$, then $\mu_M$ is entirely determined by $\tau$ and $M$ since, for all $y \in A^M$, we have
\begin{equation}
\label{e:mu-M}
\mu_M(y) = \tau(x)(1_G),
\end{equation}
where $x \in A^G$ is any configuration satisfying $x\vert_M = y$.
Moreover, every finite subset $M' \subset G$ containing $M$ is also a memory set for $\tau$ (with associated local defining map $\mu_{M'} \colon A^{M'} \to B$ given by $\mu_{M'}(y) = \mu_M(y\vert_M)$ for all $y \in A^{M'}$). In fact (see for example \cite[Section 1.5]{book}), the following holds. Every cellular automaton $\tau \colon A^G \to B^G$ admits a unique memory set $M_0 \subset G$ of minimal cardinality. Moreover, a
subset $M \subset G$ is a memory set for $\tau$ if and only
if $M_0 \subset M$. This  memory set $M_0$ is called the \emph{minimal}
memory set of $\tau$.
\par
From the definition, it easily follows that every cellular automaton  
$\tau \colon A^G \to B^G$ is $G$-equivariant, i.e., 
$$
\tau(gx) = g\tau(x) \quad \forall x \in A^G, \forall g \in G
$$
(cf. \cite[Proposition 1.4.6]{book}), and continuous with respect  to the prodiscrete topologies on $A^G$ and $B^G$ (cf. \cite[Proposition 1.4.8]{book}).

\subsection{Composition of cellular automata}
\label{ss:composition-ca}
The following result is well known. The proof we present here follows the one given 
in \cite[Remark 1.4.10]{book}.
An alternative proof will be given in Remark \ref{r:GCH-consequence} below and 
the result will be strengthened later in Proposition \ref{p:C-composition}.

\begin{proposition}
\label{p:composition}
Let $G$ be a group and let $A$, $B$ and $C$ be sets.
Suppose that  $\tau \colon A^G \to B^G$ and $\sigma \colon B^G \to C^G$ are cellular automata.
Then the composite map $\sigma \circ \tau \colon A^G \to C^G$ is a cellular automaton.
\end{proposition}

\begin{proof}
Let $S$ (resp. $T$) be a memory set for $\sigma$ (resp. $\tau$) and
let $\mu \colon B^S \to C$ (resp. $\nu \colon A^T \to B$) be the corresponding local defining map.
Then the set $ST = \{st: s \in S, t \in T\} \subset G$ is finite.
We have $sT \subset ST$ for every $s \in S$. 
Consider, for each $s \in S$,  the projection map $\pi_s \colon A^{ST} \to A^{sT}$,
 the bijective map  $f_s \colon A^{sT} \to A^T$  defined by
$f_s(y)(t) = y(st)$ for all $y \in A^{sT}$ and $t \in T$, 
and the map $\varphi_s  \colon A^{ST} \to B$ given by
\begin{equation}
\label{e:varphi-s}
\varphi_s :=\nu \circ f_s \circ \pi_s.
\end{equation}
Finally, let
\begin{equation}
\label{e:Phi}
\Phi := \prod_{s \in S} \varphi_s \colon  A^{ST} \to B^S
\end{equation}
be the product map defined by $\Phi(z)(s) = \varphi_s(z)$ for all $z \in A^{ST}$.
\par
Let us show that $\sigma \circ \tau$ is a cellular automaton with memory set  $ST$ and
associated local defining map
\begin{equation}
\label{kappa}
\kappa := \mu \circ \Phi   \colon A^{ST} \to C.
\end{equation}
 Let $x \in A^G$. We first observe that, for all   $s \in S$ and   $t \in T$, 
we  have
\begin{align*}
  (s^{-1} x)(t) 
& =  x(st) \\
& =  x\vert_{ST}(st)\\
&  = \left((f_s \circ \pi_s)\left(x\vert_{ST}\right)\right)(t),
\end{align*}
so that
\begin{equation*}
(s^{-1} x)\vert_{T} = (f_s \circ \pi_s)(x\vert_{ST}).
\end{equation*}
It follows that
\begin{equation*}
\tau( x)(s) = \nu\left((s^{-1} x)\vert_{T}\right) = \nu\left(f_s(\pi_s\left( x\vert_{ST}\right))\right) = \varphi_s\left( x\vert_{ST}\right).
\end{equation*}
Thus, we have
\begin{equation}
\label{e:mu-mu1-mu2}
\left(\tau( x)\right)\vert_{S} = \Phi \left(x\vert_{ST}\right).
\end{equation}
We finally get
\begin{align*}
\label{e;mu-tau-circ-tau}
 \left((\sigma \circ \tau)(x)\right)(g) & = \sigma(\tau(x))(g) \\
& = \mu\left((g^{-1}\tau(x))\vert_{S}\right)\\
& = \mu(\tau(g^{-1}x)\vert_{S}) && \text{(by $G$-equivariance of $\tau$)}\\
& = \mu(\Phi\left((g^{-1}x)\vert_{ST}\right)) && \text{(by \eqref{e:mu-mu1-mu2})} \\
& = \kappa\left((g^{-1}x)\vert_{ST}\right) && \text{(by \eqref{kappa})}. 
 \end{align*}
This shows that $\sigma \circ \tau$ is a cellular automaton with  memory set $ST$ and associated local defining map $\kappa$.
\end{proof}

If we fix a group $G$, we deduce from Example \ref{ex:identity-is-ca} and Proposition 
\ref{p:composition} that the cellular automata $\tau \colon A^G \to B^G$ are the morphisms of a subcategory of the category of sets whose objects are all the sets of the form $A^G$.
We shall denote this subcategory by $\CA(G)$. 

 \subsection{The Curtis-Hedlund theorem}
When the alphabet $A$ is finite, one has the \emph{Curtis-Hedlund theorem} 
\cite{hedlund} (see also \cite[Theorem 1.8.1]{book}):

\begin{theorem}
\label{t:curtis-hedlund}
Let $G$ be a group, $A$ a finite set and $B$ a set. 
Let $\tau \colon A^G \to B^G$ be a map. Then the following conditions are equivalent:
\begin{enumerate}[\rm (a)]
\item
$\tau$ is a cellular automaton;
\item
$\tau$ is $G$-equivariant and continuous (with respect  to the prodiscrete topologies on $A^G$ and $B^G$). 
\end{enumerate}
\qed
\end{theorem}

As already mentioned in Subsection \ref{ss:ca}, the implication (a) $\Rightarrow$ (b) remains true  for $A$ infinite.
When the group $G$ is non-periodic (i.e., there exists $g \in G$ of infinite order) and $A$ is infinite, one can always construct a $G$-equivariant continuous self-mapping of $A^G$ which is not a cellular automaton (see \cite{periodic} and \cite[Example 1.8.2]{book}).

\begin{example}
For $G = A = \Z$, the map $\tau \colon A^G \to A^G$, defined by
$$\tau(x)(n) = x(x(n)+n),$$
is $G$-equivariant and continuous, but $\tau$ is not a cellular automaton.
\end{example}

\subsection{Uniform spaces and the generalized Curtis-Hedlund theorem}
Let $X$ be a set.
\par 
We denote by $\Delta_X$ the diagonal in $X \times X$, that is the set
$\Delta_X = \{ (x,x) : x \in X \}$.
\par
The \emph{inverse} $\overset{-1}{U}$ of a subset $U \subset X \times X$ is the subset of $X \times X$ defined by
$\overset{-1}{U} = \{ (x,y) : (y,x) \in U \}$.
We define the \emph{composite} $U \circ V$ of two subsets $U$ and $V$ of $X \times X$  by
$$
U \circ V = \{ (x,y): \text{ there exists  } z \in X \text{  such that  } (x,z) \in U \text{  and  } (z,y) \in V \}  \subset X \times X.
$$

\begin{definition}
Let $X$ be a set. A \emph{uniform structure} on $X$ is a non--empty set $\UU$ of subsets of $X \times X$ satisfying the following conditions:
\begin{enumerate}[(UN-1)]
\item if $U \in \UU$, then $\Delta_X \subset U$;
\item if $U \in \UU$ and $U \subset V \subset X \times X$, then $V \in \UU$;
\item if $U \in \UU$ and $V \in \UU$, then $U \cap V \in \UU$;
\item if $U \in \UU$, then $\overset{-1}{U} \in \UU$;
\item if $U \in \UU$, then there exists $V \in \UU$ such that 
$V \circ V \subset U$.
\end{enumerate}
The elements of $\UU$ are then called the \emph{entourages} of the uniform structure and the set $X$ is called a \emph{uniform space}.
\end{definition}

A subset $\BB \subset \UU$ is called a \emph{base} of $\UU$ if
for each $W \in \UU$ there exists $V \in \BB$ such that $V \subset W$.

Let $X$ and $Y$ be uniform spaces. A map $f \colon X \to Y$ is called
\emph{uniformly continuous} if it satisfies the following condition: 
for each entourage $W$ of $Y$, there exists an entourage $V$ of $X$ such that $(f \times f)(V) \subset W$. Here $f \times f$ denotes the map from $X \times X$ into $Y \times Y$ defined by $(f \times f)(x_1,x_2) = (f(x_1),f(x_2))$ for all $(x_1,x_2) \in X \times X$.

If $X$ is a uniform space, there is an induced topology on $X$ characterized by the fact that the neighborhoods of an arbitrary point $x \in X$ consist of the sets 
$U[x] = \{y \in X : (x,y) \in U\}$, where $U$ runs over all entourages of $X$.
This topology is Hausdorff if and only if the intersection of all the entourages of $X$ is reduced to the diagonal $\Delta_X $. Moreover, every uniformly continuous map $f \colon X \to Y$ is continuous with respect to the induced topologies on $X$ and $Y$.

\begin{example} 
If $(X,d_X)$ is a metric space, there is a natural uniform structure on $X$ whose entourages are the sets $U \subset X \times X$ satisfying the following condition:
there exists a real number $\varepsilon >0$ such that
$U$ contains all pairs $(x,y) \in X \times X$ such that  $d_X(x,y) < \varepsilon$.
The topology associated with this uniform structure is then the same as the topology induced by the metric.
If $(Y,d_Y)$ is another metric space, then a map $f \colon X \to Y$ is uniformly continuous
if and only if for every $\varepsilon > 0$ there exists $\delta > 0$ such that
$d_Y(f(x_1),f(x_2)) < \varepsilon$ whenever $d_X(x_1,x_2) < \delta$.
\end{example}

The theory of uniform spaces was developed by A. Weil in \cite{weil-uniforme}
(see e.g. \cite{bourbaki-top-gen}, \cite{james}, \cite[Appendix B]{book}).

Let us now return back to configuration spaces and cellular automata.
Let $G$ be a group and $A$ be a set.
We equip $A^G$ with its \emph{prodiscrete uniform structure}, that
is with the uniform structure admitting the sets
$$W(\Omega) := \{(x,y) \in A^G \times A^G: x\vert_\Omega = y\vert_\Omega\},$$
where $\Omega \subset G$ runs over all finites subsets of $G$, as a base of entourages.

We then have the following extension of the Curtis-Hedlund theorem:

\begin{theorem}[\cite{periodic}; see also Theorem 1.9.1 in \cite{book}]
\label{t:ca-equiv-uc}
Let $G$ be a group and let $A$ and $B$ be sets. Let $\tau \colon A^G \to B^G$ be a map.
Then the following conditions are equivalent:
\begin{enumerate}[\rm (a)]
\item
$\tau$ is a cellular automaton;
\item
$\tau$ is $G$-equivariant and uniformly continuous (with respect to the prodiscrete uniform structures on $A^G$ and $B^G$).
\end{enumerate}
\qed
\end{theorem}

\begin{remark}
\label{r:GCH-consequence}
Since the composite of two $G$-equivariant (resp. uniformly continuous) maps is $G$-equivariant (resp. uniformly continuous), 
we immediately deduce from Theorem \ref{t:ca-equiv-uc} an alternative proof of the fact
that the composite of two cellular automata is a
cellular automaton (cf. Proposition \ref{p:composition}).
\end{remark}

\subsection{Reversible cellular automata}
\label{ss:reversible}
 Given a group $G$ and two sets $A$ and $B$, a cellular automaton $\tau \colon A^G \to B^G$ is called \emph{reversible} if $\tau$ is bijective and its inverse map $\tau^{-1} \colon B^G \to A^G$ is also a cellular automaton.
 Observe that the reversible cellular automata $\tau \colon A^G \to B^G$ are precisely the isomorphisms of the category $\CA(G)$ introduced at the end of Subsection \ref{ss:composition-ca}.
It immediately  follows from Theorem \ref{t:ca-equiv-uc}
that a bijective cellular automaton $\tau \colon A^G \to B^G$ is reversible if and only if the inverse map $\tau^{-1} \colon B^G \to A^G$ is uniformly continuous with respect to the prodiscrete uniform structures on $A^G$ and $B^G$.
\par
When the alphabet $A$ is finite, every bijective cellular automaton $\tau \colon A^G \to B^G$ is reversible by compactness of $A^G$. On the other hand, when $A$ is infinite and the group $G$ is non-periodic, there exist bijective cellular automata $\tau \colon A^G \to A^G$ that are not reversible (see \cite{periodic},  \cite[Example 1.10.3]{book}, and the examples given at the end of the present paper).
 

\subsection{Induction and restriction of cellular automata}
\label{ss:induction-restriction}
Let $G$ be a group, $A$ and $B$ two sets, and $H$ a subgroup of $G$.
\par
Suppose that a cellular automaton $\tau \colon A^G \to B^G$ admits a memory set $M$ such 
that $M \subset H$. Let $\mu_M^G \colon A^M \to B$ denote the associated local defining map.
Then the map $\tau_H \colon A^H \to B^H$ defined by
$$
\tau_H(y)(h) = \mu_M^G((h^{-1}y)\vert_M)
\quad \text{  for all  } y \in A^H \text{ and } h \in H,
$$
is a cellular automaton over the group $H$   with memory set $M$ and local defining map $\mu_M^H:=\mu_M^G$.
One says that $\tau_H$ is the cellular automaton obtained by \emph{restriction} of  
$\tau$ to $H$.
\par
Conversely, suppose that  $\sigma \colon A^H \to B^H$  is a cellular automaton with memory set $N \subset H$ and local defining map
$\nu_N^H \colon A^N \to B$. Then the map $\sigma^G \colon A^G \to B^G$ defined by
$$
\sigma^G(x)(g) = \nu_N^H((g^{-1}x)\vert_N)
\quad \text{ for all } x \in A^G \text{  and  } g \in G,
$$
is a cellular automaton over the group $G$   with memory set $N$ and local defining map $\nu_N^G := \nu_N^H$.
One says that $\sigma^G$ is the cellular automaton obtained by \emph{induction} of $\sigma$ to $G$.
\par
It immediately follows from their definitions that induction and restriction are operations one inverse to the other in the sense that one has $(\tau_H)^G = \tau$ 
and $(\sigma^G)_H = \sigma$ for every cellular automaton  $\tau \colon A^G \to B^G$ over $G$ admitting a memory set contained in $H$, and every cellular automaton $\sigma \colon A^H \to B^H$ over $H$.
We shall use the following result,
 established in \cite[Theorem 1.2]{induction}
(see also \cite[Proposition 1.7.4]{book}).

\begin{theorem}
\label{t:induction}
Let $G$ be a group, $A$ and $B$ two sets, and $H$ a subgroup of $G$. 
 Suppose that  $\tau \colon A^G \to B^G$ is a cellular automaton over $G$ admitting
a memory set contained in $H$
and let $\tau_H \colon A^H \to B^H$ denote the cellular automaton over $H$ obtained by restriction. 
Then the following holds:
\begin{enumerate}[{\rm (i)}]
 \item 
$\tau$ is injective if and only if $\tau_H$ is injective;
 \item 
$\tau$ is surjective if and only if $\tau_H$ is surjective;
 \item 
$\tau$ is bijective if and only if $\tau_H$ is bijective;
\item  
$\tau(A^G)$ is closed in $B^G$ for the prodiscrete topology if and only if $\tau_H(A^H)$ is closed in $B^H$ for the prodiscrete topology. 
\end{enumerate}
\qed
\end{theorem}

\section{Cellular automata over concrete categories}
\label{s:ca-concrete-categories}

We assume some familiarity with the basic concepts of category theory 
 (the reader is refered to \cite{maclane-category} and  \cite{abstract-and-concrete} for a detailed introduction to category theory). We adopt the following notation:
 
 \begin{itemize}
 \item
 $\Set$ is the category where objects are sets and morphisms are maps between them;
 \item
 $\Setf$ is the full subcategory of $\Set$ whose objects are finite sets;
 \item
 $\Grp$ is the category where objects are groups and morphisms are group homomorphisms between them;
 \item
 $\Rng$ is the category where objects are rings and morphisms are ring homomorphisms (all rings are assumed to be unital and ring homomorphisms are required to preserve the unity elements);
 \item
 $\Fld$ is the full category of $\Rng$ whose objects are fields (a field is a nontrivial commutative ring in which every nonzero element is invertible);
 \item
 $R$-$\Mod$ is the category where objects are left modules over a given ring $R$ and morphisms are $R$-linear maps between them;
 \item
 $R$-$\Modfg$ is the full subcategory of $R$-$\Mod$ whose objects are finitely-generated left  modules over a given ring $R$;
 \item
 $K$-$\Vect = K$-$\Mod$ is the category where objects are vector spaces over a given field $K$ and morphisms are $K$-linear maps between them;
 \item
 $K$-$\Vectfd = K$-$\Modfg$ is the full subcategory of $K$-$\Vect$ whose objects are 
 finite-dimensional vector spaces over a given field $K$;
\item
$K$-$\Aal$ is the category where objects are affine algebraic sets over a given field $K$ and morphisms are regular maps between them.
Recall that  $A$ is an \emph{affine algebraic set} over a field $K$  if $A \subset K^n$, for some integer $n \geq 0$, and $A$ is  
 the set of common zeroes of some set of polynomials $S \subset K[t_1,t_2,\dots,t_n]$.
Recall also that  a map $f \colon A \to B$  from an affine algebraic set $A \subset K^n$ to another affine algebraic set  $B \subset K^m$ is called \emph{regular} 
if $f$ is the restriction of some polynomial map $P \colon K^n \to K^m$
(see for instance \cite{bump}, \cite{kraft}, \cite{milne}, \cite{reid} for an introduction to affine algebraic geometry);
\item
$\Top$ is the category where objects are topological spaces and morphisms are continuous maps between them;
\item
$\CHT$ is the subcategory of $\Top$ whose objects are compact Hausdorff topological spaces;
\item
$\Man$ is the subcategory of $\Top$ whose objects are compact topological manifolds
(a \emph{topological manifold} is a connected Hausdorff topological space in which every point admits an open neighborhood homeomorphic to $\R^n$ for some integer $n \geq 0$).   
\end{itemize}
 
 If $\CC$ is a category and $A$ is a $\CC$-object, we shall denote by $\Id_A$ the identity morphism of $A$.

 \subsection{Concrete categories}
\label{s:cc}
 A \emph{concrete category} is a pair $(\CC,U)$, where $\CC$ is a category and
$U \colon \CC \to \Set$ is a faithful functor from $\CC$ to the category $\Set$.
The functor $U$ is called the \emph{forgetful functor} of the concrete category $(\CC,U)$.
We will  denote a concrete category $(\CC,U)$ simply by $\CC$ if its forgetful functor $U$ is clear from the context.
 \par
Let $(\CC,U)$ be a concrete category.
 If $A$ is a $\CC$-object , the set $U(A)$ is called the \emph{underlying set} of $A$.
Similarly, if $f \colon A \to B$ is a $\CC$-morphism, the map $U(f) \colon U(A) \to U(B)$ is called the \emph{underlying map} of $f$.
Note that two distinct $\CC$-objects   may have the same underlying set.
However, the faithfulness of $U$ implies that a $\CC$-morphism  is entirely determined by its underlying map once its source and target objects are given.
\par
Every subcategory   of a concrete category is itself a concrete category.
More precisely, if $(\CC,U)$ is a concrete category, $\DD$ is a subcategory of $\CC$, and $U\vert_\DD$ denotes the restriction of $U$ to $\DD$, then
$(\DD,U\vert_\DD)$ is a concrete category.
\par 
The categories $\Set$, $\Grp$,   $\Rng$, $R$-$\Mod$, $K$-$\Aal$ and  $\Top$,   equipped with their obvious forgetful functor to $\Set$, provide basic examples of concrete categories. 
On the other hand, it can be shown that the homotopy category $\hTop$,
where objects are topological spaces and morphisms are homotopy classes of continuous maps between them, is \emph{not concretizable}, in the sense that there exists no faithful functor 
$U \colon \hTop \to \Set$ (see \cite{freyd} and \cite[Exercice 5J]{abstract-and-concrete}).
\par
Let $\CC$ be a category. Recall that a \emph{product} of a family $(A_i)_{i \in I}$ of $\CC$-objects  is a pair $(P,(\pi_i)_{i \in I})$, where $P$ is a $\CC$-object   and $\pi_i \colon P \to A_i$, $i \in I$,  are $\CC$-morphisms   satisfying the following universal property: 
 if $B$ is a $\CC$-object equipped with $\CC$-morphisms $f_i \colon B \to A_i$, $i \in I$,  then there is a unique $\CC$-morphism $g \colon B \to P$ such that $f_i = \pi_i \circ g$ for all $i \in I$.
When it exists, such a product is essentially unique, in the sense that if 
$(P,(\pi_i)_{i \in I})$ and $(P',(\pi_i')_{i \in I})$ are two products of the family $(A_i)_{i \in I}$ then there exists a unique $\CC$-isomorphism
$\varphi \colon P \to P'$ such that $\pi_i = \pi_i' \circ \varphi$ for all $i \in I$.
By a common abuse, one writes $P = \prod_{i \in I} A_i$ and $g = \prod_{i \in I} f_i$. 
\par
One says that a category $\CC$ \emph{has  products} 
(cf. \cite[Definition 10.29.(1)]{abstract-and-concrete}), 
or that $\CC$ satisfies condition (P), if every set-indexed family $(A_i)_{i \in I}$ of $\CC$-objects admits a product in $\CC$.
\par
 One says that a category $\CC$ \emph{has finite products}
   (cf. \cite[Definition 10.29.(2)]{abstract-and-concrete}), 
or that $\CC$ satisfies (FP), if every finite family $(A_i)_{i \in I}$ of $\CC$-objects admits a product in $\CC$.
By means of an easy induction, one  shows that  a category $\CC$ has finite products if and only if it
satisfies the two following conditions: 
(1)  $\CC$ has a terminal object (such an object is then the product of the empty family of $\CC$-objects),
(2) any pair of $\CC$-objects admits a product (cf. \cite[Proposition 10.30]{abstract-and-concrete}).
 \par
 It is clear from these definitions that (P) implies (FP).
 \par
 The categories $\Set$, $\Grp$, $\Rng$, $R$-$\Mod$, $\Top$ and $\CHT$ all satisfy (P).
On the other hand,
the category $\Setf$ satisfies (FP) but not (P).
The category
$R$-$\Modfg$ satisfies (FP) but, unless $R$ is a zero ring,  does not satisfy (P).
The category $K$-$\Aal$ of affine algebraic sets and regular maps over a given field $K$ also satisfies (FP) but not (P).
The category $\Fld$   
 does not satisfy (FP) (it does not even admit a terminal object).
 \par
Suppose now that $(\CC,U)$ is a concrete category.
 Given a family $(A_i)_{i \in I}$ of $\CC$-objects, one says that the pair $(P,(\pi_i)_{i \in I})$ is a \emph{concrete product} of the family $(A_i)_{i \in I}$
   if $(P,(\pi_i)_{i \in I})$ is a product of the family $(A_i)_{i \in I}$ in $\CC$ and $(U(P),(U(\pi_i))_{i \in I})$ is a product of the family $(U(A_i))_{i \in I}$ in $\Set$    
   (cf. \cite[Definition 10.52]{abstract-and-concrete}).
\par
One  says that $(\CC,U)$ has \emph{concrete products}, 
 or that $(\CC,U)$ satisfies (CP), if every set-indexed family of $\CC$-objects admits a concrete product (cf. \cite[Definition 10.54]{abstract-and-concrete}).
\par 
One  says that $(\CC,U)$ has \emph{concrete finite products}, or that $(\CC,U)$ satisfies (CFP), if every finite family of $\CC$-objects admits a concrete product.
By using an induction argument, one gets a characterization of (CFP) analogous to the one given above for (FP).
More precisely, one can show that a concrete category $(\CC,U)$ has concrete finite products if and only if it
satisfies the two following conditions: 
(1)  $\CC$ has a terminal object whose underlying set is reduced to a single element 
(such an object is then the concrete product of the empty family of $\CC$-objects),
(2) any pair of $\CC$-objects admits a concrete product.
\par
It is clear from these definitions that (CP) implies (CFP).
\par
The concrete categories $\Set$, $\Grp$, $\Rng$, $R$-$\Mod$, $\Top$ and $\CHT$ all satisfy (CP).
 The concrete categories $\Setf$,  $R$-$\Modfg$ (for $R$  a nonzero ring), $K$-$\Aal$, and $\Man$ satisfy (CFP) but not  (CP), since they  do not satisfy (P).
 Here is an example of a  concrete category that satisfies (P) but not (CFP).
 
\begin{example}
Fix a set $X$ and consider the category $\CC$ defined as follows.
The objects of $\CC$ are all the pairs $(A,\alpha)$, where $A$ is a set and $\alpha \colon A \to X$ is a map.
If $(A,\alpha)$ and $(B,\beta)$ are $\CC$-objects, the morphisms from $(A,\alpha)$ to 
$(B,\beta)$ consist of all maps $f \colon A \to B$ such that $\alpha = \beta \circ f$.
It is clear that $\CC$ is a concrete category for the forgetful functor $U \colon \CC \to \Set$
that associates with each object $(A,\alpha)$ the set $A$ and with each morphism 
$f \colon (A,\alpha) \to (B,\beta)$ the map $f \colon A \to B$.
The category $\CC$ satisfies (P).
Indeed, the product of  a set-indexed family  $((A_i,\alpha_i))_{i \in I}$ of $\CC$-objects is the 
\emph{fibered product} 
$(F ,\eta)$, where
$$
F := \{(a_i)_{i \in I}  :    \alpha_i(a_i) = \alpha_j(a_j) \text{  for all  } i,j \in I \} \subset \prod_{i \in I} A_i
$$
with the natural projections maps $\pi_i \colon F \to A_i$. 
 The pair $(X,\Id_X)$, where $\Id_X \colon X \to X$ is the identity map, is clearly a terminal $\CC$-object.
 Since any terminal object in a concrete category satisfying (CFP) must be reduced to a single element,
we deduce that the concrete category $(\CC,U)$ does not satisfy (CFP) unless $X$ is reduced to a single element (observe that $(\CC,U)$ is identical to $\Set$ when $X$ is reduced to a single element). 
\end{example}

Recall that a morphism $f \colon A \to B$ in a category $\CC$ is called a \emph{retraction} if it is right-invertible, i.e., if there exists a $\CC$-morphism $g \colon B \to A$ such that 
$f \circ g = \Id_{B}$.
We have the following elementary lemma.

\begin{lemma}
\label{l:retraction-product}
Let $\CC$ be a category. Let $A$ and $B$ be two $\CC$-objects admitting a $\CC$-product $A \times B$ with first projection   $\pi \colon A \times B \to A$.
Then the following conditions are equivalent:
\begin{enumerate}[\rm (a)]
\item
$\pi$ is a retraction;
\item
there exists a $\CC$-morphism $f \colon A \to B$.
\end{enumerate}
\end{lemma}

\begin{proof}
Let $\pi' \colon A \times B \to B$ denote the second projection.
If $g \colon A \to A \times B$ is a $\CC$-morphism such that $\pi \circ g = \Id_{A}$, then $f:=\pi' \circ g$ is a $\CC$-morphism from $A$ to $B$.
This shows that (a) implies (b).
 Conversely, if (b) is satisfied then $g:= \Id_{A} \times f \colon A \to A \times B$ satisfies $\pi \circ g = \Id_{A}$. 
\end{proof}

We say that a concrete category $(\CC,U)$ satisfies (CFP+) provided  it satisfies (CFP) and the following additional condition:

\begin{itemize}
\item[(C+)]
Given any  $\CC$-object $A$ and any $\CC$-object $B$ with $U(B) \not= \varnothing$,
the first  projection morphism $\pi \colon A \times B \to A$ is a retraction.
\end{itemize}

 \begin{proposition}
 Let $(\CC,U)$ be a concrete category satisfying (CFP). Then the following conditions are equivalent:
 \begin{enumerate}[\rm (a)]
 \item
 $(\CC,U)$ satisfies (CFP+);
 \item
 for any $\CC$-object $A$ and any $\CC$-object $B$ with $U(B) \not= \varnothing$, there exists a $\CC$-morphism $f \colon A \to B$;
 \item
 if $T$ is a terminal $\CC$-object and $B$ is any  $\CC$-object with $U(B) \not= \varnothing$, then there exists a $\CC$-morphism $g \colon T \to B$.
 \end{enumerate}
 \end{proposition}

\begin{proof}
The equivalence of (a) and (b) follows from Lemma \ref{l:retraction-product}.
Condition (b) trivially implies condition (c).
To prove that (c) implies (b), it suffices to observe that by composing a $\CC$-morphism
$h \colon A \to T$ and a $\CC$-morphism $g \colon T \to B$, we get a $\CC$-morphism $f:=g \circ h \colon A \to B$. 
\end{proof}

Note that condition (c) in the preceding proposition is satisfied in particular when $\CC$ admits a
 \emph{zero object} $0$ (i.e., an object that is both initial and terminal).
 Hence every concrete category satisfying (CFP) also satisfies (CFP+) if it admits a zero object.
This is the case for  the categories $\Grp$, $\Rng$,   $R$-$\Mod$, and $R$-$\Modfg$.
\par
On the other hand, the categories $\Set$, $\Setf$, $K$-$\Aal$, $\Top$, $\CHT$, and $\Man$  also satisfy (CFP+)  although they do not admit  zero objects.
  Indeed, in any of these categories,
the only initial object is the empty one while the terminal objects are the singletons, and if $T$ is a singleton and $B$ an arbitrary object, then any map from $U(T)$ to $U(B)$ is the underlying map of a morphism from $T$ to $B$.

\subsection{Cellular automata over concrete categories}

From now on, in a concrete category, we shall use the same symbol to denote an object (resp. a morphism)   and its underlying set (resp. its underlying map).

\begin{proposition}
\label{p:carac-aca}
Let $G$ be a group and let $\CC $ be a concrete category satisfying (CFP+).
Let $A$ and $B$ be two $\CC$-objects.
Suppose that  $\tau \colon A^G \to B^G$ is a cellular automaton.
Then the following conditions are equivalent:
\begin{enumerate}[\rm (a)]
\item
there exists a memory set $M$ of $\tau$ such that 
the associated local defining map  $\mu_M \colon A^M \to B$ is a $\CC$-morphism; 
\item
for any memory set $M$ of $\tau$,
the associated local defining map  $\mu_M \colon A^M \to B$ is a $\CC$-morphism. 
\end{enumerate}
\end{proposition}

Note that, in the above statement,
$A^M$ is a $\CC$-object since $A$ is a $\CC$-object, $M$ is finite, and $\CC$ satisfies (CFP+) and hence (CFP).
On the other hand, it may happen that   the configuration spaces $A^G$ and $B^G$ 
are not $\CC$-objects (although this is the case if $\CC$ satisfies (CP)).
 
\begin{proof}[Proof of Proposition \ref{p:carac-aca}]
We can assume $A \not= \varnothing$.
 Suppose that the local defining map $\mu_M \colon A^M \to B$ is a $\CC$-morphism for some memory set $M$.  
Let $M'$ be another memory set and let us show that the associated local defining map 
$\mu_{M'} \colon A^{M'} \to B$ is also a $\CC$-morphism.
Let $M_0$ denote the minimal memory set of $\tau$.
 Recall that we have $M_0 \subset M \cap M'$.
After identifying the $\CC$-object $A^M$ (resp.~$A^{M'}$) with
the $\CC$-product $A^{M_0} \times A^{M \setminus M_0}$ (resp.~$A^{M_0} \times A^{M' \setminus M_0}$), consider the projection map $\pi \colon A^M \to A^{M_0}$ (resp.~$\pi' \colon A^{M'} \to A^{M_0}$).
We then have
\begin{equation}
\label{e;projectioon-minimal-m-s}
\mu_M = \mu_{M_0} \circ \pi \quad \text{ and } \quad \mu_{M'} = \mu_{M_0} \circ \pi'.
\end{equation}
 By condition (CFP+), the projection $\pi$ is a retraction, so that  there exists a $\CC$-morphism 
$f \colon A^{M_0} \to A^M$ such that
$\pi \circ f = \Id_{A^{M_0}}$. Using \eqref{e;projectioon-minimal-m-s}, 
we get
\[
\mu_{M'} = \mu_{M_0} \circ \pi' = \mu_{M_0} \circ \Id_{A^{M_0}} \circ \ \pi' =
\mu_{M_0} \circ \pi \circ f \circ \pi' = \mu_M \circ f \circ \pi'.
\]
Thus, the map $\mu_{M'}$ may be written as  the composite of three $\CC$-morphisms and therefore is a $\CC$-morphism.
  \end{proof}

\begin{definition}
Let $G$ be a group and let $\CC$ be a concrete category satisfying (CFP+).
Let $A$ and $B$ be two $\CC$-objects.
We say that a cellular automaton $\tau \colon A^G \to B^G$ is a 
$\CC$-\emph{cellular automaton} provided it satisfies one of the 
equivalent conditions of Proposition \ref{p:carac-aca}.
\end{definition}

\begin{example}
\label{e:identity}
Let  $G$ be a group and let $\CC$ be a concrete category satisfying (CFP+).
Let $A$ be a $\CC$-object. Then the identity map
$\Id_{A^G} \colon A^G \to A^G$   is the cellular automaton with memory set $M = \{1_G\}$ and
local defining map $\mu_M = \Id_A \colon A^M = A \to A$ 
(cf. Example \ref{ex:identity-is-ca}). As $\Id_A$ is a $\CC$-morphism, we deduce that $\Id_{A^G}$ is a $\CC$-cellular automaton.
\end{example}

\begin{example}[The Discrete Laplacian]
Let $G$ be a group, $A = \R$, and $S \subset G$ a nonempty finite subset. The map
$\Delta_S \colon \R^G \to \R^G$, defined by
$$(\Delta_S x)(g) = x(g) - \frac{1}{\vert S \vert} \sum_{s \in S} x(gs)$$
for all $x \in \R^G$ and $g \in G$, is a cellular automaton (with memory set
$M = S \cup \{1_G\}$ and associated local defining map $\mu_M \colon \R^M \to \R$
given by $\mu_M(y) = y(1_G) - \frac{1}{\vert S \vert} \sum_{s \in S} y(s)$ for
all $y \in \R^M$). Since $\R$ is a finite dimensional vector space over itself and
$\mu_M$ is $\R$-linear, we have that $\Delta_S$ is a $\CC$-cellular automaton for
$\CC = \R$-$\Vectfd$.
\end{example}

\begin{example} 
Let $G = \Z$ and let $A = K$ be any field.
Then the map $\tau \colon K^\Z \to K^\Z$, defined by
$$\tau(x)(n) = x(n + 1) - x(n)^2$$ for all $x \in K^\Z$ and $n \in \Z$,
is an $\CC$-cellular automaton (with memory set $M = \{0,1\}$ and associated local defining
map $\mu_M \colon A^M \to A$ given by $\mu(y) = y(1) - y(0)^2$ for all $y \in A^M$) for $\CC = K$-$\Aal$.
Observe that $\tau$ is not a $\CC$-cellular automaton for $\CC = K$-$\Vect$ unless
$K \cong \Z/2\Z$ is the field with two elements.
\end{example}


\begin{example}
Let $G = \Z$. Let also $\sph^1 = \R/\Z$ denote the \emph{unit circle} and, for $n \geq 1$, denote by $\tore^n = \R^n/\Z^n = (\sph^1)^n$ the $n$-\emph{torus}.
With each continuous map $f \colon \tore^{m+1} \to \sph^1$, $m \geq 0$, we can associate the cellular automaton $\tau \colon (\sph^1)^\Z \to (\sph^1)^\Z$ with memory set $M = \{0,1,\ldots,m\}$ and local defining map $\mu_M = f$. Thus we have
$$
\tau(x)(n) = f(x(n),x(n+1),\ldots, x(n+m))
$$ 
for all $x \in (\sph^1)^\Z$ and $n \in \Z$.
Then $\tau$ is a $\CC$-cellular automaton for $\CC = \Man$.
\end{example}

\begin{example}[Arnold's cat cellular automaton]
Let $G = \Z$. Let also $A = \sph^1 $ and $B = \sph^1 \times \sph^1 = \tore^2$
and consider the map $\tau \colon A^\Z \to B^\Z$ defined by
$$
\tau(x)(n) = (2x(n)+x(n+1), x(n)+x(n+1))
$$
for all $x \in A^\Z$. Then $\tau$ is a $\CC$-cellular automaton for $\CC = \Man$.
\end{example}

Given sets $A$ and $B$, a subgroup $H$ of a group $G$, and a cellular automaton 
$\tau \colon A^G \to B^G$ admitting a memory set contained in $H$, we have defined in
Subsection \ref{ss:induction-restriction} the cellular automaton $\tau_H \colon A^H \to B^H$ obtained by restriction of 
$\tau$ to $H$. We also introduced the converse operation, namely induction.
 It turns out that both restriction and induction of cellular automata preserve the property of being a   $\CC$-cellular automaton. More precisely, we have the following result. 

\begin{proposition}
\label{p:C-induction}
Let $G$ be a group and let $H$ be a subgroup of $G$. Let also $\CC$ be a category satisfying (CFP+), and let $A$ and $B$ be two 
$\CC$-objects.
Suppose that  $\tau \colon A^G \to B^G$ is a cellular automaton over $G$ admitting
a memory set contained in $H$.
Let $\tau_H \colon A^H \to B^H$ denote the cellular automaton over $H$ obtained by restriction. 
Then  $\tau$ is a $\CC$-cellular automaton if and only if $\tau_H$ is a $\CC$-cellular automaton.
 \end{proposition}

\begin{proof}
If $M$ is a memory set of $\tau$ contained in $H$, then $M$ is also a memory set for $\tau_H$. Moreover,
$\tau$ and $\tau_H$ have the same associated local defining map 
$\mu_M \colon A^M \to B$ (cf. Subsection \ref{ss:induction-restriction}). 
Therefore, the statement follows immediately from the definition of a $\CC$-cellular automaton. 
\end{proof}

\begin{proposition}
\label{p:global-CC-equiv-local}
Let $G$ be a group and let $\CC$ be a concrete category satisfying (CP) and (C+) (and hence (CFP+)).
Let $A$ and $B$ be $\CC$-objects and let $\tau \colon A^G \to B^G$ be a cellular automaton.
Then the following conditions are equivalent:
\begin{enumerate}[\rm (a)]
\item
$\tau$ is a $\CC$-morphism;
\item
$\tau$ is a $\CC$-cellular automaton.
\end{enumerate}
\end{proposition}

Note that, in the preceding statement,   $A^G$ and $B^G$ are $\CC$-objects since $\CC$ satisfies (CP).

\begin{proof}[Proof of Proposition \ref{p:global-CC-equiv-local}]
Let $M$ be a memory set for $\tau$.
We then have
\begin{equation}
\label{e:local}
(\tau(x))(g) = (\mu_M \circ \pi_M)(g^{-1}x)  \quad \forall x \in A^G, \forall g \in G,
\end{equation}
where $\pi_M \colon A^G \to A^M$ is the projection morphism. 
By definition,  $\tau$ is a $\CC$-cellular automaton if and only if  $\mu_M$ is a $\CC$-morphism.
\par
Suppose first that $\mu_M$ is a $\CC$-morphism.
For each $g \in G$, the self-map of $A^G$ given by $x \mapsto g^{-1}x$ is a $\CC$-morphism since it just permutes coordinates of the $\CC$-product $A^G$. On the other hand, the projection $\pi_M \colon A^G \to A^M$ is also a $\CC$-morphism.
Therefore, we deduce from \eqref{e:local} that the map from $A^G$ to $B$ given by $x \mapsto \tau(x)(g)$ is a $\CC$-morphism for each $g \in G$.
It follows that $\tau \colon A^G \to B^G$ is a $\CC$-morphism.
\par
Conversely, suppose that $\tau$ is a $\CC$-morphism.
Let us show that $\tau$ is a $\CC$-cellular automaton.
We can assume $A \not= \varnothing$.
Denote by $p \colon A^G \to A^{\{1_G\}} = A$ the projection $\CC$-morphism $x \mapsto x(1_G)$.
By applying \eqref{e:local}
with $g = 1_G$, we get
\begin{equation}
\label{e:local_at-id}
(p \circ \tau)(x)  = \mu_M(y)
\end{equation} 
for all  $x \in A^G$ and $y \in A^M$ with $x\vert_M = y$.
As $\CC$ satisfies (C+), the projection $\CC$-morphism $\pi \colon A^G = A^M \times A^{G \setminus M} \to A^M$ is a retraction. Therefore there exists a $\CC$-morphism 
$f \colon A^M \to A^G$ such that $\pi \circ f = \Id_{A^M}$.
We then deduce  from \eqref{e:local_at-id} that $\mu_M = p \circ \tau \circ f$. Consequently, $\mu_M$ is a $\CC$-morphism. This shows that $\tau$ is a $\CC$-cellular automaton. 
\end{proof}

\begin{examples}
1) Take $\CC = R$-$\Mod$.
Given two left $R$-modules $A$ and $B$, a cellular automaton  $\tau \colon A^G \to B^G$ is a 
$\CC$-cellular automaton if and only if $\tau$ is $R$-linear with respect to the product 
$R$-module structures on $A^G$ and $B^G$.
\par
2) Take $\CC = \Top$.
Given two topological spaces $A$ and $B$, a cellular automaton $\tau \colon A^G \to B^G$ is 
a $\CC$-cellular automaton if and only if $\tau$ is continuous with respect to the product topologies on $A^G$ and $B^G$ (in general, these topologies are coarser than the prodiscrete topologies).
\end{examples}
 
\begin{proposition}
\label{p:C-composition}
Let $G$ be a group. Let $\CC$ be a concrete category satisfying (CFP+), and  let $A,B$ and $C$ be $\CC$-objects. 
Suppose that  $\tau \colon A^G \to B^G$ and $\sigma \colon B^G \to C^G$ 
are $\CC$-cellular automata. Then $\sigma \circ \tau \colon A^G \to C^G$ is a $\CC$-cellular automaton.
\end{proposition}

\begin{proof}
We have already seen (cf. Proposition \ref{p:composition} and Remark \ref{r:GCH-consequence}) that $\sigma \circ \tau$ is a
cellular automaton. 
Thus we are only left to show that $\sigma \circ \tau$
admits a local defining map which is a $\CC$-morphism.
Let $S$ (resp. $T$) be a memory set for $\sigma$ (resp. $\tau$) and let
$\mu \colon B^S \to C$ (resp. $\nu \colon A^T \to B$) denote the corresponding local defining map. Then, as we showed in the proof of Proposition \ref{p:composition}, the set $ST$ is a memory set for $\sigma \circ \tau$ and the map $\kappa \colon A^{ST} \to C$ defined by \eqref{kappa} is the corresponding local defining map.
 Now, since $\tau$ is a $\CC$-cellular automaton, we have that $\nu$ is a $\CC$-morphism.
Moreover, the maps $\pi_s \colon A^{ST} \to A^{sT}$ and  $f_s \colon A^{sT} \to A^T$  
are $\CC$-morphisms for every $s \in S$. 
As $\varphi_s = \nu \circ f_s \circ \pi_s $ (cf. \eqref{e:varphi-s}), 
 it follows that the product map  
 $\Phi = \prod_{s \in S} \varphi_s \colon A^{ST} \to B^S$ (cf. \eqref{e:Phi}) is a $\CC$-morphism.
Since $\sigma$ is a $\CC$-cellular automaton, its local defining map $\mu \colon B^S \to C$ is also a $\CC$-morphism and therefore the map $\kappa = \mu \circ \Phi \colon A^{ST} \to C$ is a $\CC$-morphism as well.
This completes the proof that $\sigma \circ \tau$ is a $\CC$-cellular automaton.
\end{proof}
 
Let $G$ be a group and $\CC$ a concrete category  satisfying (CFP+).
Then it follows from Example \ref{e:identity} and Proposition \ref{p:C-composition} that
there is   a category $\CA(G,\CC)$ 
having the same objects as $\CC$,
in which the identity morphism of an object $A$ is the identity map $\Id_{A^G} \colon A^G \to A^G$, and where the 
morphisms from an object $A$ to an object $B$ are the $\CC$-cellular automata 
$\tau \colon A^G \to B^G$ (with composition of morphisms given by the usual composition of maps). 
The category $\CA(G,\CC)$ is a concrete category when equipped with the functor $U \colon \CA(G,\CC) \to \Set$ given by $U(A) = A^G$ and $U(\tau) = \tau$. 
Observe that the image of
the functor $U$ is a subcategory of the category $\CA(G)$  defined at the end of Subsection
\ref{ss:composition-ca}.

\section{Projective sequences of sets}
\label{sec:proj-seq}

 Let us briefly recall some elementary facts about projective sequences of sets and their projective limits.
\par
 A \emph{projective sequence of sets} is a sequence $(X_n)_{n \in \N}$ of sets equipped with maps $f_{nm} \colon X_m \to X_n$, defined for all $n,m \in \N$ with $ m \geq n$,
and satisfying the following conditions:

\begin{enumerate}[(PS-1)]
\item
$f_{n n}$ is the identity map on $X_n$ for all $n \in \N$;
\item
$f_{n k} = f_{n m} \circ f_{m k}$ for all $n,m,k \in \N$ such that $k \geq m \geq n$.
\end{enumerate}
We shall denote such a projective sequence by $(X_n,f_{n m})$ or simply by $(X_n)$.
\par
Observe that the projective sequence $(X_n,f_{n m}$) is entirely determined by the maps
$g_n = f_{n,n+1} \colon X_{n + 1} \to X_n$, $n \in \N$, since
\begin{equation}
\label{e:gives-fnm-gn}
f_{n m} = g_n \circ g_{n + 1} \circ \ldots \circ g_{m - 1} 
\end{equation}
for all $m > n$. 
Conversely, if we are given a sequence of maps $g_n \colon X_{n + 1} \to X_n$, $n \in \N$, then there is a unique projective sequence $(X_n, f_{n m})$ satisfying \eqref{e:gives-fnm-gn}.
\par
 The \emph{projective limit} $X = \varprojlim X_n$ of the projective sequence of sets $(X_n,f_{n m})$ is the subset  $X \subset \prod_{n \in \N} X_n$ consisting of the sequences $x = (x_n)_{n \in \N}$ satisfying 
$x_n = f_{n m}(x_m)$ for all $n,m \in \N$ such that $m \geq n$
(or, equivalently, $x_n = g_n(x_{n + 1})$ for all $n \in \N$, where $g_n = f_{n,n+1}$).
Note that there is a canonical map $\pi_n \colon X \to X_n$ sending $x$ to $x_n$ and that one has
$\pi_n = f_{n m} \circ \pi_m$ for all $m,n \in \N$ with $m \geq n$. 
\par
The fact that the projective limit $X = \varprojlim X_n$ is not empty clearly implies that all the sets $X_n$ are nonempty.
However, it can happen that the projective limit $X = \varprojlim X_n$ is empty even if all the sets $X_n$ are nonempty.
The following statement yields a sufficient condition for the projective limit to be nonempty.

 \begin{proposition}
\label{p:proj-limit-nonempty}
Let $(X_n,f_{n m})$ be a projective sequence of sets and
 let $X = \varprojlim X_n$ denote its projective limit.
Suppose that all maps $f_{n m} \colon X_m \to X_n$,   $m \geq n$ are surjective.
Then all canonical maps $\pi_n \colon X \to X_n$, $m \in \N$, are surjective.
In particular, if in addition $X_{n_0} \not= \varnothing$ for some $n_0 \in \N$, 
then one has $X \not= \varnothing$. 
\end{proposition}

\begin{proof}
Let $n \in \N$ and $x_n\in X_n$. As the maps $f_{k,k + 1}$, $k \geq n$, are surjective, we can construct by induction a sequence 
$(x_k)_{k \geq n}$ such that $x_k = f_{k,k + 1}(x_{k + 1})$ for all $k \geq n$.
Let us set $x_k = f_{k n}(x_n)$ for $k< n$.
Then the sequence $x = (x_k)_{k \in \N}$  is in $X$ and satisfies $x_n = \pi_n(x)$. 
This shows that $\pi_n$ is surjective.
\end{proof}

\begin{remarks}
1) For the maps $f_{n m}$, $m \geq n$, to be surjective, it suffices that all the maps 
$f_{n,n+1}$ are  surjective.
\par
2)  When the maps $f_{n m}$ are all surjective, the following conditions are equivalent:
(1) there exists $n_0 \in \N$ such that $X_{n_0} \not= \varnothing$, and (2)
one has $X_n \not= \varnothing$ for all $n\in \N$. 
  \end{remarks}

Let $(X_n,f_{n m})$ be a projective sequence of sets.
 Property (PS-2) implies that, for each $n \in \N$, 
 the sequence of subsets $f_{nm}(X_m) \subset X_n$, $m \geq n$, is non-increasing. Let us set, for each $n \in \N$, 
$$
X_n' = \bigcap_{m \geq n} f_{nm}(X_m) \subset X_n.
$$
The set $X_n'$ is called the set of \emph{universal elements} of $X_n$ 
(cf. \cite{grothendieck-ega-3}).
Observe that $f_{n m}(X_m') \subset X_n'$ for all $m \geq n$.
Thus, the map $f_{n m}$ induces by restriction a map $f_{n m}' \colon X_m' \to X_n'$ 
for all $m  \geq n$. 
Clearly $(X_n',f_{n m}')$ is a projective sequence.
This projective sequence is called the
\emph{universal projective sequence} associated with the projective sequence $(X_n,f_{n m})$.

\begin{proposition}
\label{p:same-proj-limit}
Let $(X_n,f_{n m})$ be a projective sequence of sets and let $(X_n',f_{n m}')$ be the associated universal projective sequence. Then one has
\begin{equation}
\label{e:proj-univ}
\varprojlim X_n = \varprojlim X_n'.
\end{equation}
\end{proposition}
\begin{proof}
Let us set $X = \varprojlim X_n$ and $X' = \varprojlim X_n'$.
Since $X_n' \subset X_n$ and $f_{nm}'$ is the restriction of $f_{nm}$ to $X_n'$, for all
$n,m \in \N$ with $m \geq n$, we clearly have $X' \subset X$.
To show the converse inclusion, let $x = (x_n)_{n \in \N} \in X$. 
We have $x_n = f_{nm}(x_m)$ for all $n,m \in \N$ such that
$m \geq n$, so that $x_n \in \bigcap_{m \geq n} f_{nm}(X_m) = X_n'$. 
Since $f_{nm}'(x_n) = f_{nm}(x_n)$, we then deduce that $X \subset X'$.
This shows \eqref{e:proj-univ}.
\end{proof}

\begin{corollary}
\label{cor:IP}
Let $(X_n,f_{n m})$ be a projective sequence of sets.
Suppose that the following conditions are satisfied:
\begin{enumerate}[\rm ({I}P-1)]
\item
there exists $n_0 \in \N$ such that $\bigcap_{k \geq n_0} f_{n_0 k}(X_k) \not= \varnothing$;
\item
for all $n,m  \in \N$ with $m \geq n$ and  all $x_n' \in \bigcap_{i \geq n}f_{n i}(X_i)$, one has 
$$
\bigcap_{j \geq m} f_{n m}^{-1}(x_n') \cap f_{m j}(X_j) \not= \varnothing.
$$
\end{enumerate}
Then  one has
$\varprojlim X_n \not= \varnothing$. 
\end{corollary}

\begin{proof}
Consider the universal projective sequence  $(X_n',f_{n m}')$ associated with the projective sequence $(X_n,f_{n m})$.
Observe that condition (IP-1) says that $X_{n_0}' \not= \varnothing$.
On the other hand, condition (IP-2) says that for all $n,m \in \N$ with $m \geq n$, 
one has $f_{n m}^{-1}(x_n') \cap X_m' \neq \varnothing$ for all $x_n' \in X_n'$, i.e., that the map $f_{n m}'$ is surjective.
Thus, by applying Proposition \ref{p:same-proj-limit} and Proposition \ref{p:proj-limit-nonempty},
we get $\varprojlim X_n = \varprojlim X_n' \not= \varnothing$.  
\end{proof}

\begin{remark}
\label{rem:non-empty-decreasing}
Let $(X_n,f_{n m})$ be an arbitrary sequence of sets.
We claim that, given $m \geq n$ and 
$x_n' \in \bigcap_{i \geq n}f_{n i}(X_i)$,  one has
$f_{n m}^{-1}(x_n') \cap f_{m j}(X_j) \not= \varnothing$ for every $j \geq m$.
Indeed, since $x_n' \in f_{n j}(X_j)$, we can find $y_j \in X_j$ such that $x_n' = f_{n j}(y_j)$. 
Setting $z_m = f_{m j}(y_j)$, we then have
$f_{n m}(z_m) = f_{n m} \circ f_{m j}(y_j) = x_n'$, so that
$z_m \in f_{n m}^{-1}(x_n') \cap f_{m j}(X_j) $.
This proves our claim. It follows that the sets
$f_{n m}^{-1}(x_n') \cap f_{m j}(X_j)$ form a non-increasing sequence of nonempty subsets of $X_m'$.
Condition (IP-2) says that the intersections of this sequence is not empty for all $m \geq n$ and $x_n' \in X_n'$. 
\end{remark}

\section{Algebraic and subalgebraic subsets}

\subsection{Algebraic subsets}

 \begin{definition}
Let $\CC$ be a concrete category.
Given a $\CC$-object $A$, we say that a subset $X \subset A$ is  $\CC$-\emph{algebraic} 
(or simply \emph{algebraic} if there is no ambiguity on the category $\CC$)
if $X$ is the inverse image of a point by some $\CC$-morphism, i.e.,
if there exist a $\CC$-object $B$, a point $b \in B$, and a $\CC$-morphism $f \colon A \to B$ such that $X = f^{-1}(b)$.
\end{definition}

\begin{remark}
If $\CC$ is a concrete category admitting a terminal object which is reduced to a single element (this is for example the case when $\CC$ is a concrete category satisfying (CFP+)),
then every   $\CC$-object $A$ is  a $\CC$-algebraic subset of itself.
Indeed, we then have $A = f^{-1}(t)$, where $f \colon A \to T$ is the unique $\CC$-morphism from $A$ to $T$ and $t$ is the unique element of $T$.
 \end{remark}

\begin{remark}
Suppose that $\CC$ is a concrete category satisfying (CFP) and that $A$ is a $\CC$-object. Then the set of $\CC$-algebraic subsets of $A$ is closed under finite intersections. 
Indeed, if $(X_i)_{i \in I}$ is a finite family of $\CC$-algebraic subsets of a $\CC$-object $A$, we can find $\CC$-morphisms
$f_i \colon A \to B_i$  and points $b_i \in B_i$  such that $X_i = f_i^{-1}(b_i)$. 
Then $\bigcap_{i \in I} X_i = f^{-1}(b)$, where $f = \prod_{i \in I} f_i \colon A \to B$, 
$B = \prod_{i \in I} B_i$ and $b = (b_i)_{i \in I}$. 
\end{remark}

\begin{examples}
1) In the category $\Set$ or in its full subcategory $\Setf$,  the algebraic subsets of an object $A$ consist of all the subsets of $A$.
\par
2) In the category $\Grp$, the algebraic subsets of an object $G$ consist of the empty set $\varnothing$ and all  the left-cosets (or, equivalently, all the right-cosets) of the normal subgroups of $G$, i.e., the subsets of the form $gN$, where $g \in G$ and $N$ is a
normal subgroup of $G$.
\par
3) In the category $\Rng$, the algebraic subsets of an object $R$ consist of $\varnothing$ and all the translates of the two-sided ideals of $R$, i.e., 
the subsets of the form $r + I$, where $r \in R$ and $I$ is a two-sided ideal of $R$.
\par
4) In the category $\Fld$,  the algebraic subsets of an  object $K$ 
are $\varnothing$   and all the singletons $\{k\}$, $k \in K$.
\par
5) In the category $R$-$\Mod$, the algebraic subsets of an object $M$ are $\varnothing$ and all the translates of the submodules of $M$, i.e., the subsets of the form $m + N$, where $m \in M$ and $N$ is a submodule of $M$.
\par
6) In the category $\Top$, every subset of an object $A$ is algebraic.
Indeed, if $A$ is a topological space and $X$ is a nonempty subset of  $A$,
then $X = f^{-1}(b_0)$, where $B$ is the quotient space of $A$ obtained by identifying $X$ to a single point $b_0$ and $f \colon A \to B$ is the quotient map.
\par 
7)  In the full subcategory of $\Top$ whose objects consist of all the normal Hausdorff spaces,
the algebraic subsets of an object $A$ are precisely the closed subsets of $A$.
\par
8) In the category $K$-$\Aal$ of affine algebraic sets over a field $K$, the algebraic subsets of an object $A$ are precisely the subsets of $A$ that are closed in the Zariski topology. If $A \subset K^n$, these subsets are the algebraic subsets of $K^n$ (in the usual sense of algebraic geometry) that are contained in $A$.
\end{examples}

\subsection{Subalgebraic subsets}
\begin{definition}
Let $\CC$ be a concrete category.
Given a $\CC$-object $B$, we say that a subset $Y \subset B$ is $\CC$-\emph{subalgebraic} (or simply
\emph{subalgebraic} if there is no ambiguity on the category $\CC$)
if $Y$ is the image of some $\CC$-algebraic subset by some $\CC$-morphism, i.e., 
if there exist a $\CC$-object $A$, a $\CC$-algebraic subset $X \subset A$, and a $\CC$-morphism $f \colon A \to B$ such that $Y = f(X)$. 
\end{definition} 

Note that, if $\CC$ is a concrete category and $A$ is a $\CC$-object,
then every $\CC$-algebraic subset $X \subset A$ is also $\CC$-subalgebraic since 
$X = \Id_A(X)$. 
Note also that if $g \colon B \to C$ is a $\CC$-morphism and $Y$ is a $\CC$-subalgebraic subset of $B$, then $g(Y)$ is a $\CC$-subalgebraic subset of $C$.

\begin{examples}
1) In the category  $\Set$ or in the category  $\Setf$, the subalgebraic subsets of an object $A$ coincide with its algebraic subsets,  that is, they consist of all the subsets of $A$.
\par
2) In the category $ \Grp$, every subalgebraic subset of an object $G$ is either empty or   of the form $gH$, where $g \in G$ and $H$ is a (not necessarily normal) subgroup of $G$.
Observe that every subgroup $H \subset G$ is  subalgebraic  since it is the image of the inclusion  morphism  $\iota \colon H \to G$.
This shows in particular that there exist subalgebraic subsets that are not algebraic.
On the other hand, a group may contain subgroup cosets  which are not subalgebraic.
Consider for example the symmetric group $G = S_3$.
Then, if we take $g = (12)$ and $H = \langle (13) \rangle = \{1_G,(13)\}$, the coset $gH = \{(12),  (123) \}$ is not a subalgebraic subset of $G$.
Otherwise, there would exist a group  $G'$,  a  normal subgroup  $H' \subset G'$, 
an element $g' \in G'$,  and a group homomorphism $f \colon G' \to G$   such that $gH = f(g'H') = f(g') f(H')$. 
If  $f$ was surjective, then $f(H')$ would be normal in $G$ and thus would
have $1$, $3$ or $6$ elements, which would contradict the fact  that $H$ has order $2$.
Therefore,  the subgroup $f(G')$ must  have either
$1$, $2$ or $3$ elements. 
But this is also impossible since the coset $gH = f(g')f(H')$ has two elements and does not contain $1_G$. 
 \par
 3) In the category $\Rng$, there are subalgebraic subsets that are not algebraic.
 For example, in the polynomial ring $\Z[t]$, the subring $\Z$ is subalgebraic but not algebraic.
\par
4) In the category $\Fld$, the subalgebraic subsets of an object $K$ are
its algebraic subsets, i.e.,   $\varnothing$   and all the singletons $\{k\}$, $k \in K$.
\par
5)  In the category $R$-$\Mod$, the subalgebraic subsets of an object $M$ coincide with the algebraic subsets of $M$. Thus the subalgebraic subsets of $M$
consist of $\varnothing$ and the translates of the submodules of $M$.
 \par
 6) In the category $\Top$, the subalgebraic subsets of an object $A$ coincide with its algebraic subsets, i.e., they consist of all the subsets of $A$.
 \par
7) In the full subcategory of $\Top$ whose objects are Hausdorff topological spaces,   the open interval $]0,1[ \subset \R$ is subalgebraic but not algebraic.
\par
8)  In the category  $\CHT$ of compact Hausdorff topological spaces,
the subalgebraic subsets of an object $A$ coincide with its algebraic subsets,   i.e., are  the   closed subsets of $A$.
 \par
9) In the category $K$-$\Aal$ of affine algebraic sets over an algebraically closed field $K$,
it follows from Chevalley's theorem (see e.g. \cite[AG Section 10.2]{borel-algebraic} or 
\cite[Theorem 10.2]{milne})
that every subalgebraic subsets of an object $A$ is
 \emph{constructible}, 
that is, a finite union of subsets of the form $U \cap V$, where $U \subset A$   is  open  and 
$V \subset A$  is closed  for the Zariski topology.  
\end{examples}

\subsection{The subalgebraic intersection property}
The following definition is due to Gromov \cite[Subsection 4.C']{gromov-esav}.

\begin{definition}
 Let $\CC$ be a concrete category.
We say that $\CC$ satisfies the \emph{subalgebraic intersection
property}, briefly (SAIP), if for every $\CC$-object $A$,  every $\CC$-algebraic subset $X \subset A$, and
every non-increasing sequence $(Y_n)_{n \in \N}$ of $\CC$-subalgebraic subsets of $A$ with $X \cap Y_n \not= \varnothing$ for all $n \in \N$, 
 one has  $\bigcap_{n\in\N} X \cap Y_n \not= \varnothing$.
\end{definition}

Let us introduce one more definition.

\begin{definition}
We say that a concrete category $\CC$ is \emph{Artinian}
if the subalgebraic subsets of any $\CC$-object satisfy the descending chain condition, i.e.,  if, given any $\CC$-object $A$
and any non-increasing sequence $(Y_n)_{n \in \N}$ of $\CC$-subalgebraic subsets of $A$,   there exists $n_0 \in \N$ such that $Y_n = Y_{n + 1}$ for all $n \geq n_0$.
\end{definition}

\begin{proposition}
\label{p:art-saip}
Every Artinian concrete category satisfies (SAIP).
\end{proposition}

\begin{proof}
Let $\CC$ be an Artinian concrete category. 
Let $A$ be a $\CC$-object, $X$ an algebraic subset of $A$, and $(Y_n)_{n \in \N}$ a non-increasing sequence of subalgebraic subsets of $A$ with $X \cap Y_n \not= \varnothing$ 
for all $n \in \N$.
As $\CC$ is Artinian, we can find $n_0 \in \N$ such that
$Y_m = Y_{n_0}$ for all $m \geq n_0$. We then have
$$
\bigcap_{n \in \N} X \cap Y_n = X \cap Y_{n_0} \not= \varnothing.
$$
This shows that $\CC$ satisfies (SAIP).
\end{proof}
 
Observe that if a concrete category $\CC$ satisfies (SAIP) (resp. is Artinian), then every subcategory of $\CC$ satisfies (SAIP) (resp. is Artinian).

\begin{examples}
\label{ex:SAIP}
1) The category $\Set$ does not satisfy (SAIP). Indeed, if we take $A = X = \N$ and 
$Y_n = \{m \in \N: m \geq n\}$, $n \in \N$, we have $X \cap Y_n \neq \varnothing$ for all $n \in \N$ but $\bigcap_{n\in\N} X \cap Y_n = \varnothing$.
\par
2) The category $\Setf$ is Artinian and hence  satisfies (SAIP).
Indeed, every non-increasing sequence of subsets of a finite  set  eventually stabilizes. 
\par
3) Let $R$ be a nonzero ring and let $\CC = R$-$\Mod$. 
Consider the $R$-module $M = \bigoplus_{i \in \N} R$ and, for
every $n \in \N$, denote by $\pi_n \colon M \to R^{n+1}$ the projection map defined by $\pi_n(m) = (m_0,m_1,\ldots, m_n)$ for all $m = (m_i)_{i \in \N} \in M$. Let $y_n := (1,1,\ldots,1) \in R^{n+1}$ and set $X_n := \pi_n^{-1}(y_n)=\{m = (m_i)_{i \in \N} \in M: m_0 = m_1 = \cdots = m_n = 1\}$.
Then $X_n$ is a nonempty $\CC$-algebraic subset of $M$ and we have $X_0 \subset X_1 \supset X_2 \supset \dots$ but
$\bigcap_{n \in \N} X_n = \varnothing$. This shows that $R$-$\Mod$ does not satisfy (SAIP)
unless $R$ is a zero ring.
\par
4) Since $\Z$-$\Mod$ is a subcategory of $\Grp$, namely the full subcategory of $\Grp$ whose
objects are Abelian groups, we deduce that $\Grp$ does not satisfy (SAIP) either.
\par
5) If $K$ is a field, then the category $K$-$\Vect = K$-$\Mod$
does not satisfy (SAIP) since
any field is a nonzero ring.
\par 
6)  Given an integer $a \geq 2$, the subsets $X_n \subset \Z$ defined by
$$X_n = 1 + a + a^2 + \cdots + a^n + a^{n+1}\Z,$$ 
$n \in \N$, are nonempty $\CC$-algebraic subsets of $\Z$ for $\CC = \Rng$ and  $\CC = \Z$-$\Modfg$. 
As $X_0 \supset X_1 \supset X_2 \supset \dots$ and $\bigcap_{n \in \N} X_n = \varnothing$ 
(cf. the remark following the proof of Proposition~2.2 in \cite{artinian}), this shows that the categories
$\Rng$ and $\Z$-$\Modfg$ do not satisfy (SAIP). 
\par
7) Let $R$ be a ring. Recall that a left module $M$ is called \emph{Artinian} if the submodules of $M$ satisfy the descending chain condition (see for instance \cite[Chapter~8]{hungerford}). It is clear that, in an Artinian module, the translates of the submodules also satisfy the descending chain condition (cf. \cite[Proposition~2.2]{artinian}).
It follows that the full subcategory $R$-$\ModArt$ of $R$-$\Mod$, whose objects consist of all  the Artinian left $R$-modules, is Artinian and hence satisfies (SAIP).
Note that $R$-$\ModArt$ satisfies (CFP+) since the direct sum of two Artinian modules is itself Artinian. 
If the ring $R$ is \emph{left-Artinian} (i.e., Artinian as a left module over itself), then every finitely generated left module over $R$ is Artinian (Theorem 1.8 in  \cite[Chapter~8]{hungerford}), so that the category $R$-$\Modfg$ is Artinian and hence satisfies (SAIP) (cf. \cite{artinian}).
 \par
8) If $K$ is a field, then the category $K$-$\Vectfd = K$-$\ModArt$ satisfies (SAIP).
\par
9) The category $\Fld$ clearly satisfies (SAIP).
\par 
10) The category $\Top$ clearly does not satisfy (SAIP).
In fact, even the full subcategory of $\Top$ whose objects are Hausdorff topological spaces
does not satisfy (SAIP) since, in this subcategory,  the open intervals $(0,1/n) \subset \R$ form a non-increasing sequence of subalgebraic subsets of $\R$ with empty intersection.
\par
11)  The  category $\CHT$ of  compact Hausdorff topological spaces (and therefore its full subcategory $\Man$) satisfies (SAIP). 
This follows from the fact that, in a compact space, any family of closed subsets with the finite intersection property has a nonempty intersection. 
Observe that neither $\CHT$ nor $\Man$ are Artinian
since the arcs $X_n := \{\mathrm{e}^{\mathrm{i}\theta} : 0 \leq \theta \leq 1/(n + 1) \}$, $n \in \N$, form a non-increasing sequence of closed subsets of the unit circle
$\sph^1 := \{z \in \C : \vert z \vert = 1 \}$ which does not stabilize.
\par
12) Given a field $K$, the category $K$-$\Aal$ is Artinian if and only if  $K$ is finite.
To see this, first observe that if $K$ if finite then all objects in $K$-$\Aal$ are finite and hence 
$K$-$\Aal$ is Artinian. Then suppose that $K$ is infinite and choose a sequence $(a_n)_{n \in \N}$ of distinct elements in $K$.
Now observe that  $Y_n = K \setminus \{a_0,a_1,\dots,a_n\}$ is a nonempty 
$K$-$\Aal$-subalgebraic subset of $K$ since it is the projection on the $x$-axis of the affine algebraic curve $X_n \subset K^2$ with  equation
$(x - a_0)(x- a_1) \dots (x - a_n)y = 1$.   
As the sequence  $Y_n$ is non-increasing and does not stabilize, this shows that $K$-$\Aal$ is not Artinian.
\par
13) If $K$ is an infinite countable field, then the category $K$-$\Aal$ does not satisfy (SAIP).
Indeed, suppose that $K = \{a_n : n \in \N\}$.
Then the subsets $Y_n = K \setminus \{a_0,a_1,\dots,a_n\}$ form a non-increasing sequence of nonempty subsets of $K$ with empty intersection.
As each $Y_n$ is subalgebraic in $K$ (see the preceding example), this shows that $K$-$\Aal$ does not satisfy (SAIP).
\par
14) If $K$ is an uncountable algebraically closed field, then the category $K$-$\Aal$ of affine algebraic sets over  $K$ satisfies (SAIP) 
(see \cite[Subsection 4.C'']{gromov-esav} and \cite[Proposition 4.4]{algebraic}).
\par
15) The category $\CC = \R$-$\Aal$ of real affine algebraic sets does not satisfy (SAIP).
Indeed, the subsets $X_n = [n,+\infty)$, $n \in \N$, are nonempty 
$\CC$-subalgebraic subsets of $\R$ since $X_n = P_n(\R)$ for $P_n(t) = t^2 + n$.
On the other hand, we have $X_0 \supset X_1 \supset X_2 \supset \cdots$ but $\bigcap_{n \in \N} X_n = \varnothing$.
\end{examples}

\subsection{Projective sequences of algebraic sets}

Let $\CC$ be a concrete category satisfying condition (CFP+).
We say that a projective sequence $(X_n,f_{n m})$ of sets is a \emph{projective sequence of $\CC$-algebraic sets} provided there is a projective sequence $(A_n,F_{n m})$, consisting of $\CC$-objects $A_n$ and $\CC$-morphisms $F_{n m} \colon A_m \to A_n$ for all $n,m \in \N$ such that $m \geq n$, satisfying the following conditions:

\begin{enumerate}[(PSA-1)]
\item
$X_n$ is a $\CC$-algebraic subset of $A_n$ for every $n \in \N$;
\item
$F_{n m}(X_m) \subset X_n$ and $f_{n m}$ is the restriction of $F_{n m}$ to $X_m$ for all $m,n \in \N$ such that $m \geq n$.
\end{enumerate} 

Note that $f_{n m}(X_m) = F_{n m}(X_m)$ in (PSA-2) above is a $\CC$-subalgebraic subset
of $A_n$ for all $n,m \in \N$ such that $m \geq n$.

The following result constitutes an essential ingredient in the proof of Theorem \ref{t:closed-image}.

\begin{theorem}
\label{t:proj-limit-sub} 
Let $\CC$ be a concrete category satisfying  (SAIP).
Suppose that  $(X_n,f_{n m})$ is a projective sequence of nonempty $\CC$-algebraic sets.
Then one has $\varprojlim X_n \not= \varnothing$. 
\end{theorem}

\begin{proof}
Let $(A_n,F_{n m})$ be a projective sequence of $\CC$-objects and morphisms 
satisfying conditions (PSA-1) and (PSA-2) above.
Let $n \in \N$. 
 For all $k \geq n$, the image set $f_{n k}(X_k) = F_{n k}(X_k)$, 
being the image of a nonempty $\CC$-algebraic subset under a $\CC$-morphism, is a nonempty $\CC$-subalgebraic subset of $A_n$. 
As the sequence $f_{n k}(X_k)$, $k = n, n + 1, \dots$,   is non-increasing, 
we deduce from (SAIP) that
\begin{equation}
\label{e;IP-1}
X_n' = \bigcap_{k \geq n}  f_{n k}(X_k) \not= \varnothing
\end{equation}
for all $n \in \N$.
\par
Let us fix $m,n \in \N$ with $m \geq n$ and   $x_n' \in X_n'$.
In Remark \ref{rem:non-empty-decreasing},
  we observed that 
$f_{n m}^{-1}(x_n') \cap f_{m j}(X_j) \not= \varnothing$
for all $j \geq m$.    
On the other hand, we have
$$
f_{n m}^{-1}(x_n') \cap f_{m j}(X_j) = F_{n m}^{-1}(x_n') \cap F_{m j}(X_j).
$$
As   the sets $F_{m j}(X_j)$, for $j = m, m + 1, \dots$, form a non-increasing sequence of $\CC$-subalgebraic subsets of $A_m$ and $F_{n m}^{-1}(x_n')$ is a $\CC$-algebraic subset of $A_m$, we get 
 $$
\bigcap_{j \geq m} f_{n m}^{-1}(x_n') \cap f_{m j}(X_j) \not= \varnothing
$$
by applying  (SAIP) again.
 This is condition (IP-2) in Corollary \ref{cor:IP}. Since (IP-1) follows from \eqref{e;IP-1}, Corollary \ref{cor:IP} ensures that $\varprojlim X_n \not= \varnothing$.
\end{proof}
\section{The closed image property}
\label{sec:closed-image}

One says that a map $f \colon X \to Y$ from a set $X$ into a topological space $Y$ has the 
\emph{closed image property}, briefly (CIP) (cf. \cite[Subsection 4.C'']{gromov-esav}), if its image $f(X)$ is closed in $Y$.

\begin{theorem}
\label{t:closed-image}
Let $\CC$ be a concrete category satisfying conditions (CFP+) and (SAIP).
Let $G$ be a group. 
Then every $\CC$-cellular automaton $\tau \colon A^G \to B^G$
satisfies   (CIP) with respect to the prodiscrete topology on $B^G$.
\end{theorem}

\begin{proof}
Let $\tau \colon A^G \to B^G$ be a $\CC$-cellular automaton.
Let $M \subset G$ be a memory set for $\tau$ and let $\mu_M \colon A^M \to B$ denote the associated local defining map.
\par
Suppose first that the group $G$ is countable.
Then we can find a sequence $(E_n)_{n \in \N}$ of finite subsets of $G$ such that $G = \bigcup_{n \in \N} E_n$,
$M \subset E_0$,  and $E_n \subset E_{n + 1}$ for all $n \in \N$.
Consider, for each $n \in \N$, the finite subset $F_n \subset G$ defined by  
$F_n = \{g \in G: gM \subset E_n\}$. 
Note that $G = \bigcup_{n \in \N} F_n$,
$1_G \in F_0$,  and $F_n \subset F_{n + 1}$ for all $n \in \N$.
\par
It follows from \eqref{e;local-property} that if $x$ and $x'$ are elements in $A^G$ such that $x$ and $x'$ coincide on $E_n$, then the configurations $\tau(x)$ and $\tau(x')$ coincide on $F_n$.  Therefore, we can define a map
$\tau_n \colon A^{E_n} \to B^{F_n}$ by setting 
$$
\tau_n(u) = (\tau(x))\vert_{F_n}
$$ 
for all $u \in A^{E_n}$, where $x \in A^G$ denotes an arbitrary configuration extending $u$.
Observe that both $A^{E_n}$ and $B^{F_n}$ are $\CC$-objects  as they are finite Cartesian powers of the $\CC$-objects $A$ and $B$ respectively. 
\par
We claim that $\tau_n$ is a $\CC$-morphism for every $n \in \N$. Indeed, let $n \in \N$. For every $g \in F_n$, we have $gM \subset E_n$. Denote, for each $g \in F_n$,  by $\pi_{n,g} \colon A^{E_n} \to A^{gM}$
the projection $\CC$-morphism. Consider also, for all $g \in G$,  the $\CC$-isomorphism
$\phi_g \colon A^{gM} \to A^M$  
defined by $\left(\phi_g(u)\right)(m) = u(gm)$ for all  $m \in M$. 
Then, for each $g \in F_n$, the map $F_g := \mu_M \circ \phi_g \circ \pi_{n,g} \colon A^{E_n} \to B$   is a $\CC$-morphism since it is the composite of $\CC$-morphisms.
Observe that
$F_g(x) = \tau(x)(g)$ for all $x \in A^G$. This shows that  $\tau_n = \prod_{g \in F_n} F_g$ and the claim  follows.
\par
Let now $y \in B^G$ and suppose that $y$ is in the closure of $\tau(A^G)$ for the prodiscrete topology on $B^G$.
Then, for every $n \in \N$, we can find $z_n \in A^G$ such that 
\begin{equation}
 \label{e;y-z-n-B-n}
y\vert_{F_n} =   (\tau(z_n))\vert_{F_n}.
\end{equation}
Consider, for each $n \in \N$, the $\CC$-algebraic subset $X_n \subset A^{E_n}$ defined by 
$X_n = \tau_n^{-1}(y\vert_{F_n})$.
We have $X_n \not= \varnothing$ for all $n \in \N$  since  $z_n\vert_{E_n} \in X_n$ 
by \eqref{e;y-z-n-B-n}. 
Observe that, for  $m \geq n$, the projection $\CC$-morphism  $F_{n m} \colon A^{E_m} \to A^{E_n}$   induces by restriction a map $f_{nm} \colon X_m \to X_n$. 
Conditions (PSA-1) and (PSA-2) are trivially satisfied so that
$(X_n,f_{nm})$ is a projective sequence of nonempty $\CC$-algebraic sets.
Since $\CC$ satisfies (SAIP),
 we have $\varprojlim X_n \not= \varnothing$ by Theorem \ref{t:proj-limit-sub}.
Choose an element $(x_n)_{n \in \N} \in \varprojlim X_n$. Thus
$x_n \in A^{E_n}$  and $x_{n + 1}$ coincides with $x_n$ on $E_n$ for all $n \in \N$. 
As $G = \cup_{n \in \N} E_n$, we deduce that there exists a (unique) configuration 
$x \in A^G$ such that $x\vert_{E_n} = x_n$ for all $n \in \N$. 
Moreover, we have $\tau(x)\vert_{F_n}= \tau_n(x_n) = y_n = y\vert_{F_n}$ for all $n$ since $x_n \in X_n$. 
As $G = \cup_{n \in \N} F_n$, this shows that $\tau(x) = y$.
This completes the proof that $\tau$ satisfies condition (CIP) in the case when $G$ is countable.
\par
Let us treat now the case of an arbitrary (possibly uncountable) group $G$.
Let $H$ denote the subgroup of $G$ generated by the memory set $M$.
Observe that $H$ is countable since $M$ is finite.
The restriction cellular automaton $\tau_H \colon A^H \to B^H$ is a $\CC$-cellular automaton by Proposition \ref{p:C-induction}.
Thus, by the first part of the proof, 
$\tau_H$ satisfies condition (CIP), that is,  
$\tau_H(A^H)$ is closed in $B^H$ for the prodiscrete topology on $B^H$.  
By applying Theorem \ref{t:induction}.(iii), we deduce that $\tau(A^G)$ is also closed in $B^G$ for the prodiscrete topology on $B^G$.  
Thus $\tau$ satisfies condition (CIP).
\end{proof}

From Theorem \ref{t:closed-image} and Examples \ref{ex:SAIP} we recover
results from \cite{gromov-esav}, \cite{garden}, \cite[Lemma 3.2]{artinian}, \cite{induction}, \cite{periodic}, and \cite{algebraic}. 

\begin{corollary}
\label{c:closed-image}
Let $G$ be a group. Then every $\CC$-cellular automaton $\tau \colon A^G \to B^G$
satisfies (CIP) with respect to the prodiscrete topology on  $B^G$, when $\CC$ is one of the following concrete categories:
\begin{itemize}
\item 
$\Setf$, the category of finite sets;
\item 
$K$-$\Vectfd$, the category of finite-dimensional vector spaces over an arbitrary field $K$;
\item 
$R$-$\ModArt$, the category of left Artinian modules over an arbitrary ring $R$;
\item 
$R$-$\Modfg$, the category of finitely generated left modules over an arbitrary left-Artinian ring $R$;
\item 
$K$-$\Aal$, the category of affine algebraic sets over an arbitrary uncountable algebraically closed field $K$;
\item
$\CHT$, the category of compact Hausdorff topological spaces;
\item 
$\Man$, the category of compact topological manifolds.
\end{itemize}
\end{corollary}

\begin{remark}
When $\CC$ is  $\Setf$,   we can directly deduce that any $\CC$-cellular automaton $\tau \colon A^G \to B^G$ satisfies (CIP) from the compactness of $A^G$, the continuity of $\tau$, and the fact that $B^G$ is Hausdorff. 
\end{remark} 
 
\begin{remarks}
1) It is shown in \cite{periodic} (see also \cite[Example 8.8.3]{book}) that if $G$ is a 
non-periodic group and $A$ is an infinite set, then there exists a cellular automaton
$\tau \colon A^G \to A^G$ whose image $\tau(A^G)$ is not closed in $A^G$ for the prodiscrete
topology.
\par
2) Let $K$ be a field and let $\CC = K$-$\Vect$. It is shown in \cite{periodic} (see also \cite[Example 8.8.3]{book}) that is $G$ is a non-periodic group and $A$ is an infinite-dimensional vector space over $K$, then there exists a $\CC$-cellular automaton $\tau \colon A^G \to A^G$ whose image $\tau(A^G)$ is not closed in $A^G$ for the prodiscrete topology.
\par
3) In \cite[Remark 5.2]{algebraic}, it is shown that if $G$ is a non-periodic group, then there exists a $\R$-$\Aal$-cellular automaton 
$\tau \colon \R^G \to \R^G$ whose image is not closed in $\R^G$ for the prodiscrete topology. 
\end{remarks}

\section[Surjunctivity of $\CC$-cellular automata]{Surjunctivity of $\CC$-cellular automata over residually finite groups}
\label{sec:surjunctivity-rf-groups}

\subsection{Injectivity and surjectivity in concrete categories}
Let $\CC$ be a category and $f \colon A \to B$ a $\CC$-morphism.
We recall that $f$ is said to be a $\CC$-\emph{monomorphism}
provided that for any two $\CC$-morphisms $g_1, g_2 \colon C \to A$ 
the equality $f \circ g_1 = f \circ g_2$ implies $g_1 = g_2$. 
We also recall that $f$ is called a $\CC$-\emph{epimorphism}
if for any two $\CC$-morphisms 
$h_1, h_2 \colon B \to C$ the equality $h_1 \circ f = h_2 \circ f$ implies
$h_1 = h_2$. In other words, a $\CC$-monomorphism (resp. $\CC$-epimorphism) is a left-cancellative (resp. right-cancellative) $\CC$-morphism.
\par
Suppose now that $\CC$ is a concrete category. 
Then one says that a $\CC$-morphism $f \colon A \to B$ is \emph{injective} 
(resp. \emph{surjective}, resp. \emph{bijective}) if (the underlying map of) $f$ is injective (resp. surjective, resp. bijective) in the set-theoretical sense.
It is clear that every injective (resp.~surjective) $\CC$-morphism is a $\CC$-monomorphism 
(resp.~a $\CC$-epimorphism). 
In concrete categories such as $\Set$, $\Setf$, $\Grp$, $R$-$\Mod$, $R$-$\Modfg$, $\Top$, $\CHT$, or $\Man$   the converse is also true so that the class of 
 injective (resp. surjective) morphisms coincide with the class of monomorphisms (resp. epimorphisms) in these categories
(the fact that every epimorphism is surjective in $\Grp$ is a non-trivial result, 
see \cite{linderholm}).
 However, there exist concrete categories admitting monomorphisms (resp. epimorphisms) that fail to be injective (resp. surjective).
 
\begin{examples}
1) Let $\CC $ be the full subcategory of $\Grp$ consisting of all divisible Abelian groups. Recall that an Abelian group $G$ is called  \emph{divisible} if for each $g \in G$ and each integer $n \geq 1$, there is an element $g' \in G$ such that $ng' = g$. 
Then the quotient map $f \colon \Q  \to \Q/\Z$ is a non-injective $\CC$-monomorphism.
\par
2)  Let $\CC = \Rng$.  Then the inclusion map $f \colon \Z \to \Q$ is a non-surjective $\CC$-epimorphism.
\par
3) Let $\CC$ be the full subcategory of $\Top$ whose objects are Hausdorff spaces.
Then the inclusion map $f \colon \Q \to \R$ is a non-surjective $\CC$-epimorphism.
\end{examples}

\subsection{Surjunctive categories}
\begin{definition}
A concrete category $\CC$ is said to be \emph{surjunctive} if  every injective 
$\CC$-endomorphism $f \colon A \to A$ is surjective (and hence bijective).
\end{definition}

\begin{examples}
\label{ex:surjunctive}
1) The category $\Set$ is not surjunctive but $\Setf$ is.  
Indeed,  a set $A$ is finite if and only if every injective map $f \colon A \to A$ is surjective
(Dedekind's characterization of infinite sets).
\par
2) The map $f \colon \Z \to \Z$, defined by $f(n) = 2n$ for all $n \in \Z$, is injective but not surjective. This shows that the categories $\Grp$, $\Z$-$\Mod$ and   $\Z$-$\Modfg$  are not surjunctive. 
\par
3) Let $R$ be a nonzero ring.
Then the  map $f \colon R[t] \to R[t]$, defined by $P(t) \mapsto P(t^2)$, is   injective but not surjective.
As $f$ is both a ring and a $R$-module endomorphism, 
this shows that the categories $\Rng$ and $R$-$\Mod$ are not surjunctive.
\par
4) If $k$ is a field and $k(t)$ is the field of rational functions on $k$, then the map
$f \colon k(t) \to k(t)$, defined by $F(t) \mapsto F(t^2)$, is a field homomorphism which is injective but not surjective. This shows that the category $\Fld$ is not surjunctive. 
\par
5) Let $K$ be a field. Then the category $ K$-$\Vect$ is not surjunctive but $ K$-$\Vectfd$ is. Indeed, it is well known from basic linear algebra that for a vector space $A$ over $K$ one has $\dim_K(A)< \infty$ if and only if every injective $K$-linear map $f \colon A \to A$ is surjective.
\par
6) If $R$ is a ring then the category $R$-$\ModArt$ of Artinian left-modules over $R$
is surjunctive (see e.g. \cite[Proposition 2.1]{artinian}).
\par
7) Let $R$ be a left Artinian ring. Then the category $R$-$\Modfg$ of finitely-generated
left $R$-modules over $R$ is surjunctive (cf. \cite[Proposition~2.5]{artinian}).
\par
8) In \cite{vasconcelos}, it is  shown that, for a  commutative ring
$R$, the category $R$-$\Modfg$ is surjunctive if and only if all prime ideals in $R$ are maximal (if $R$ is a nonzero ring, this amounts to saying that $R$ has Krull dimension $0$).
The non-commutative rings $R$ such that $R$-$\Modfg$ is surjunctive are characterized 
in \cite{armendariz}.
 \par
9) Let $K$ be a field. If $K$ is algebraically closed, then the category $K$-$\Aal$ of
affine algebraic sets over $K$ is surjunctive: this is a particular case of the \emph{Ax-Grothendieck theorem} \cite[Theorem C]{ax-elementary}, \cite{ax-injective}, and \cite[Proposition 10.4.11]{grothendieck-ega-3} (see also \cite{borel-basel}).
\par
When the ground field $K$ is not algebraically closed, the category $K$-$\Aal$ may fail to be surjunctive. For instance, the injective polynomial map $f \colon \Q \to \Q$ 
defined by $f(x) = x^3$ is not surjective since $2 \notin f(\Q)$. This shows that
the category $\Q$-$\Aal$ is not surjunctive.
If $k$ is any field of characteristic $p > 0$ (e.g., $k = \Z/p\Z$) and we denote by $K = k(t)$ the field of rational functions with coefficients in $k$ in one indeterminate $t$, then, the map $f \colon k(t) \to k(t)$, defined by $f(R) = R^p$ for all $R \in k(t)$, is injective (it is the \emph{Frobenius endomorphism} of the field $k(t)$) but   not surjective since there is no $R \in k(t)$ such that $t = R^p$. Thus, the category $k(t)$-$\Aal$ is not surjunctive for any field $k$ of characteristic $p > 0$.
\par
 10) The categories $\Top$ and $\CHT$ are not surjunctive.
Indeed, if we consider the unit interval $[0,1] \subset \R$,
the continuous map $f \colon [0,1] \to [0,1]$, defined by $f(x) = x/2$ for all $x \in [0,1]$, is injective but not surjective. 
\par
11) Let $M$ be a compact topological manifold and suppose that $f \colon M \to M$ is 
an injective continuous map. Then $f(M)$ is open in $M$ by Brouwer's {\it invariance of domain} and closed by compactness of $M$. Since $M$ is connected, we deduce
that $f(M) = M$. This shows that the category $\Man$ is surjunctive.
\end{examples}

\subsection{Surjunctive groups}
 
\begin{definition}
 Let $\CC$ be a concrete category satisfying condition (CFP+).
One  says that a group $G$ is $\CC$-\emph{surjunctive} if every injective $\CC$-cellular
automaton $\tau \colon A^G \to A^G$ is surjective. In other words, 
$G$ is $\CC$-surjunctive if the category $\CA(G,\CC)$ of $\CC$-cellular automata over $G$ is surjunctive.
\end{definition}

\begin{remarks} 
1) The trivial group is $\CC$-surjunctive if and only if the category 
$\CC$ is surjunctive. 
\par
2)  There exist no $\CC$-surjunctive groups in the case when the category $\CC$ is not surjunctive.
Indeed, if $f \colon A \to A$ is a $\CC$-morphism which is injective but not surjective and $G$ is any group, then the map $\tau \colon A^G \to A^G$, defined by $\tau(x)(g) = f(x(g))$ for all $x \in A^G$ and $g \in G$,  is a $\CC$-cellular automaton (with memory set $\{1_G\}$) which is injective but not surjective.
\par
3) A group $G$ is $\Setf$-surjunctive if and only if, for any finite alphabet $A$, every injective cellular automaton
$\tau \colon A^G \to A^G$ is surjective. 
Gottschalk \cite{gottschalk} called such a group a \emph{surjunctive group}.
\end{remarks}

\subsection{Residually finite groups}
Recall that a group $G$ is called \emph{residually finite} if the intersection of its finite index subgroups is reduced to the identity element.
This is equivalent to saying that if $g_1$ and $g_2$ are distinct elements in $G$, then we can find a finite group $F$ and a group homomorphism $f \colon G \to F$ such that $f(g_1) \not= f(g_2)$.
All finite groups, all free groups, all finitely generated nilpotent groups (and hence all finitely generated abelian groups, e.g. $\Z^d$ for any $d \in \N$),
and all fundamental groups of compact topological manifolds of dimension $\leq 3$    
 are residually finite. 
 All finitely generated linear groups are residually finite by a theorem of Mal'cev.
 On the other hand, the additive group $\Q$, the group of permutations of $\N$, 
 the Baumslag-Solitar group  $BS(2,3) = \langle a,b : a^{-1}b^2a = b^3 \rangle$, and all infinite simple groups provide examples of groups which are not residually finite.
\par
The following dynamical characterization of residual finiteness is well known (see e.g. \cite[Theorem 2.7.1]{book}):

\begin{theorem}
\label{t:characterization-rf}
Let $G$ be a group. Then the following conditions are equivalent:
\begin{enumerate}[{\rm (a)}]
\item 
the group $G$ is residually finite;
\item 
for every set $A$, the set of points of $A^G$ which have a finite $G$-orbit is dense
in $A^G$ for the prodiscrete topology.
\end{enumerate}
\qed
\end{theorem}

\subsection{Surjunctivity of residually finite groups}

\begin{theorem}
\label{t:C-surjunctivity-rf}
Let $\CC$ be a concrete category satisfying conditions (CFP+) and (SAIP). Suppose
that $\CC$ is surjunctive. Then every residually finite group is $\CC$-surjunctive.
In other words, every injective $\CC$-cellular automaton $\tau \colon A^G \to A^G$ is surjective  when $G$ is residually finite (e.g., $G = \Z^d$). 
\end{theorem}

Before proving Theorem \ref{t:C-surjunctivity-rf}, let us introduce some additional notation.
\par
Let $A$, $M$, and $N$ be sets.
Suppose that we are given a map $\rho \colon M \to N$.
Then $\rho$ induces a map $\rho^* \colon A^N \to A^M$ defined by $\rho^*(y) = y \circ \rho$ 
for all $y \in A^N$.

\begin{lemma}
\label{l:rho-induces-morphism}
Let $\CC$ be a concrete category satisfying condition (CFP).
Let $A$ be a $\CC$-object and suppose that we are given a map $\rho \colon M \to N$, where $M$ and $N$ are finite sets.
Then the induced map $\rho^* \colon A^N \to A^M$ is a $\CC$-morphism.
\end{lemma}

\begin{proof}
We have $\rho^*(y)(m) = y(\rho(m))$ for all $m \in M$ and $y \in A^N$. 
Thus, if we denote by $\pi_n \colon A^N \to A$, $n \in N$, the $\CC$-morphism given by the projection map on the $n$-factor,  we have $\rho^* = \prod_{m \in M} \pi_{\rho(m)}$. Consequently, $\rho^*$ is a $\CC$-morphism.
\end{proof}

\begin{proof}[Proof of Theorem \ref{t:C-surjunctivity-rf}]
Let $G$ be a residually finite group and suppose that $\tau \colon A^G \to A^G$ is an injective 
$\CC$-cellular automaton. For every finite index subgroup $H$ of $G$ we denote by  
$\Fix(H)$ the subset of $A^G$ consisting of all configurations
$x \in A^G$ that are fixed by $H$, that is, such that $hx = x$ for all $h \in H$.  
We also denote by $H \backslash G = \{Hg : g \in G \}$ the finite set consisting of all right cosets of $H$ in $G$ and by $\rho_H \colon G \to H \backslash G$ the canonical surjective map sending  
each $g \in G$ to $Hg$.
Consider the induced map $\rho_H^*\colon A^{H \backslash G} \to A^G$.
One immediately checks that  $\rho_H^*(A^{H \backslash G}) \subset \Fix(H)$. In fact,
the map $\rho_H^* \colon A^{H \backslash G} \to \Fix(H)$ is bijective  (see e.g. \cite[Proposition 1.3.3]{book}).
Observe now that by $G$-equivariance of $\tau$ we have $\tau(\Fix(H)) \subset \Fix(H)$.
Denote by $\sigma:=\tau\vert_{\Fix(H)} \colon \Fix(H) \to \Fix(H)$ the map obtained by restricting $\tau$ to   $\Fix(H)$ and let
$\widetilde{\sigma} \colon A^{H \backslash G} \to A^{H \backslash G}$ be the conjugate of $\sigma$ by $\rho^*_H$, that is, the map given by
$\widetilde{\sigma} = (\rho_H^*)^{-1} \circ \sigma \circ \rho_H^*$.
We claim that $\widetilde{\sigma}$ is a $\CC$-morphism.
To see this, it suffices to prove that,
for each $t \in H \backslash G$, the map $\pi_t \colon A^{H \backslash G} \to A$ defined by
$\pi_t(y) = \widetilde{\sigma}(y)(t)$ is a $\CC$-morphism, since then 
$\widetilde{\sigma} = \prod_{t \in T} \pi_t$. Choose a memory set $M$ for $\tau$ and let $\mu_M \colon A^M \to A$ denote the associated local defining map.
For $t =gH \in T$, consider the map  $\psi_t \colon M \to H \backslash G$  defined by $\psi_t(m) = \rho_H(gm)$ for all $m \in M$. It is obvious that $\psi_t$ is well defined
(i.e. it does not depend on the particular choice of the representative $g \in G$ of
the coset $t = Hg$). If $\psi^* \colon A^{H \backslash G} \to A^M$ is the induced map, 
we then have $\pi_{t} = \mu_M \circ \psi_t^*$. But $\mu_M$ is a $\CC$-morphism since $\tau$ is a $\CC$-cellular automaton.  On the other hand, $\psi_t^*$ is also a $\CC$-morphism
by Lemma \ref{l:rho-induces-morphism}. It follows that $\pi_t$ is a $\CC$-morphism, proving our claim.
Now observe that $\sigma \colon \Fix(H) \to \Fix(H)$ is injective since it is the restriction of $\tau$.  
As $\widetilde{\sigma}$ is conjugate to $\sigma$, we deduce that $\widetilde{\sigma}$ is injective as well. Since by our assumptions the category $\CC$ is surjunctive, we deduce that $\widetilde{\sigma}$ is surjective.
Thus, $\sigma$ is also surjective and hence $\Fix(H) = \sigma(\Fix(H)) = \tau(\Fix(H)) \subset \tau(A^G)$.
\par
Let $E \subset A^G$ denote the set of configurations whose orbit under the $G$-shift is finite. Then we have
$$
E = \bigcup_{H \in \FF} \Fix(H) \subset \tau(A^G),
$$
where $\FF$ denotes the set of all finite index subgroups of $G$.
On the other hand, the residual finiteness of $G$ implies that $E$ is dense in $A^G$ (cf. Theorem \ref{t:characterization-rf}). As $\tau(A^G)$ is closed in $A^G$ by Theorem \ref{t:closed-image}, we conclude that $\tau(A^G) = A^G$.
\end{proof}

From Theorem \ref{t:C-surjunctivity-rf}, Examples \ref{ex:surjunctive} and Examples \ref{ex:SAIP} we deduce the following:

\begin{corollary}
\label{cor:residually-finite-C-surj}
All residually finite groups are $\CC$-surjunctive when $\CC$ is one of the following concrete
categories:
\begin{itemize}
\item 
$\Setf$, the category of finite sets;
\item 
$K$-$\Vectfd$, the category of finite-dimensional vector spaces over an arbitrary field $K$;
\item 
$R$-$\ModArt$, the category of left Artinian modules over an arbitrary ring $R$;
\item 
$R$-$\Modfg$, the category of finitely generated left modules over an arbitrary left-Artinian 
ring $R$;
\item 
$K$-$\Aal$, the category of affine algebraic sets over an arbitrary uncountable algebraically closed field $K$;
\item 
$\Man$, the category of compact topological manifolds.
\end{itemize}
\end{corollary}

\begin{remarks} 
1) Let $\CC = \Setf$. The $\CC$-surjunctivity of residually finite groups was established by 
W.~Lawton \cite{gottschalk} (see also \cite[Theorem 3.3.1]{book}).
As mentioned in the Introduction, all amenable groups are $\CC$-surjunctive (cf. \cite[Corollary 5.9.3]{book}).
These results were generalized by Gromov \cite{gromov-esav} and Weiss \cite{weiss-sgds}
(see also \cite[Theorem 7.8.1]{book}) who proved that all \emph{sofic} groups are $\CC$-surjunctive.  It is not known whether all groups are $\CC$-surjunctive (resp. sofic) or not.
\par
2) Let $\CC = K$-$\Vectfd$, where $K$ is an arbitrary field.
In \cite{garden} (see also \cite{gromov-esav}) we proved that residually finite groups and
amenable groups are $\CC$-surjunctive. More generally, in \cite{israel} (see also \cite{gromov-esav} and \cite[Theorem 8.14.4]{book}) we proved that all sofic groups are $\CC$-surjunctive. We also proved (see \cite{israel} and \cite[Corollary 8.15.7]{book}) that a group $G$ is $\CC$-surjunctive, if and only if the group ring $K[G]$ is \emph{stably finite}, that is, the following condition holds: if two square matrices $a$ and $b$ with entries in the group ring $K[G]$ satisfy $ab = 1$, then they also satisfy $ba=1$. 
We recall that   Kaplansky (cf. \cite{kaplansky}) conjectured that all group rings are stably finite. 
He proved the conjecture when the ground field $K$ has characteristic zero, but for positive characteristic, though proved for all sofic groups by Elek and Szabo \cite{ES} (see also \cite{israel} and \cite[Corollary 8.15.8]{book}), the Kaplansky conjecture remains open. 
In other words, it is not known whether all groups are $\CC$-surjunctive or not when $K$ has positive characteristic.
\par
3) In \cite[Corollary 1.3]{modsofic}, it is shown that if $R$ is a left-Artinian ring and $\CC = R$-$\Modfg$, then every sofic group is $\CC$-surjunctive.
 \par
4) Let $\CC = K$-$\Aal$, where $K$ is an uncountable algebraically closed field.
The fact that all residually finite groups are $\CC$-surjunctive was established in \cite[Corollary 1.2]{algebraic} (see also \cite{gromov-esav}).
We do not know how to prove that all amenable (resp. all sofic) groups are $\CC$-surjunctive.
\end{remarks}

\section{Reversibility of $\CC$-cellular automata}
\label{sec:reversibility}

\subsection{The subdiagonal intersection property}

Let $\CC$ be a concrete category satisfying (CFP) and let $f \colon A \to B$ be a $\CC$-morphism.
Consider the map $g \colon A \times A \to B \times B$ defined by $g(a_1,a_2) = (f(a_1),f(a_2))$ for all $(a_1,a_2) \in A \times A$.
If $\pi_i^A \colon A\times A \to A$ and $\pi_i^B \colon B \times B \to B$, $i = 1,2$, denote the projection morphisms,
 we have   $\pi_i^B \circ g = f \circ \pi_i^A$ for $i = 1,2$.
Therefore, the maps $\pi_1^B \circ g$ and $\pi_2^B \circ g$ are $\CC$-morphisms.
We deduce that  $g$ is also a $\CC$-morphism.
We shall write this morphism $g = f \times f$ and call it the \emph{square} of $f$ (not to be confused with the product map $h \colon A \to B \times B$ defined by
$h(a) = (f(a),f(a))$ for all $a \in A$).

\begin{definition}
Let $\CC$ be a concrete category satisfying condition (CFP). Let $A$ be a $\CC$-object.
\par
A subset $X \subset A \times A$ is called \emph{$\CC$-square-algebraic} 
if it is the inverse image of a point by the square of some $\CC$-morphism, i.e., if
there exists a $\CC$-morphism $f \colon A \to B$ and an element $(b_1,b_2) \in B \times B$
such that 
$$
X = \{(a_1,a_2) \in A \times A: f(a_1) = b_1 \text{  and  } f(a_2) = b_2\}.
$$
\par
A subset $X \subset A \times A$
is called \emph{$\CC$-prediagonal}  if there exists a $\CC$-morphism
$f \colon A \to B$ such that 
$$
X = \{(a_1,a_2) \in A \times A: f(a_1) = f(a_2)\}.
$$
In other words, a subset $X \subset A \times A$ is $\CC$-prediagonal  if and only if $X$  is the inverse image of the diagonal $\Delta_B \subset B \times B$  by the square 
$f \times f \colon A \times A \to B \times B$
of some $\CC$-morphism $f \colon A \to B$. 
\par 
A subset $X \subset A \times A$
is called \emph{$\CC$-codiagonal} if there exists a $\CC$-morphism
$f \colon A \to B$ such that 
$$
Y = \{(a_1,a_2) \in A \times A: f(a_1) \neq f(a_2)\}.
$$
In other words, a subset of $A \times A$ is $\CC$-codiagonal if and only if it is  the complement of a $\CC$-diagonal subset.
\par
A subset $X \subset A \times A$ is called $\CC$-subdiagonal if 
it is  the image of a $\CC$-prediagonal subset by the square of a $\CC$-morphism.
In other words, a subset $X \subset A \times A$ is subdiagonal if and only if there exist $\CC$-morphisms $f \colon B \to A$ and $g \colon B \to C$ such that
$$
X = \{(f(b_1),f(b_2)) : (b_1,b_2) \in B \times B \text{  and  } g(b_1) = g(b_2) \}.
$$
\end{definition}

\begin{remarks}
Suppose that $\CC$ is a concrete category satisfying (CFP) and that $A$ is a $\CC$-object. Then:
\par
1) Every $\CC$-square-algebraic subset $X \subset A \times A$ is a $\CC$-algebraic subset
of $A \times A$.
\par
2) The set $A \times A$ is a $\CC$-square-algebraic subset of itself.
Indeed, if $T = \{t\}$ is a terminal $\CC$-object, we have 
$A \times A = \{(a_1,a_2) \in A \times A : f(a_1) = t \text{ and }  f(a_2) = t \}$, where $f \colon A \to T$ denotes the unique $\CC$-morphism from $A$ to $T$.
\par
3)  Every prediagonal subset $X \subset A \times A$ contains the 
diagonal $\Delta_A \subset A \times A$.
Moreover,  $\Delta_A $ is a
$\CC$-prediagonal subset of $A \times A$ since 
$\Delta_A = \{(a_1,a_2) \in A \times A: f(a_1) = f(a_2)\}$ for $f = \Id_A$.
 \par 
4) The set  $A \times A$ is a $\CC$-prediagonal subset of itself. Indeed if $T$ is a terminal 
$\CC$-object  and $f \colon A \to T$ is the unique $\CC$-morphism from $A$ to $T$, then 
we have $A \times A = \{(a_1,a_2) \in A \times A: f(a_1) = f(a_2)\}$.
\par
5) The set  of $\CC$-prediagonal subsets of $A \times A$ is closed
under finite intersections. Indeed, if $(X_i)_{i \in I}$ is a finite family of
$\CC$-prediagonal subsets of $A \times A$,  we can find
$\CC$-morphisms $f_i \colon A \to B_i$ such that $X_i = \{(a_1,a_2) \in A \times A:
f_i(a_1) = f_i(a_2)\}$. Then if we set $B = \prod_{i \in I} B_i$ and $f = \prod_{i \in I} f_i \colon A \to B$, we have $\bigcap_{i \in I} X_i =  \{(a_1,a_2) \in A \times A:
f(a_1) = f(a_2)\}$.
 \end{remarks}

As usual, we shall sometimes omit the letter ``$\CC$" in the words $\CC$-square-algebraic, $\CC$-prediagonal, $\CC$-codiagonal and $\CC$-subdiagonal when the ambient category is clear from the context.

\begin{examples}
1) Let $\CC = \Set$. 
Given a set $A$, the  square-algebraic subsets of $A \times A$ are precisely the subsets of the form $E \times F$, where $E$ and $F$ are arbitrary subsets of $A$. 
A subset $X \subset A \times A$
is prediagonal if and only if it is the graph of an equivalence relation on $A$.
\par
2) In the category $\Grp$, given a group $G$, a subset of $G \times G$ is square-algebraic if and only if it is either empty or of the form
$(g_1N) \times (g_2N)$, where $N$ is a normal subgroup of $G$. 
The prediagonal subsets of $G \times G$ 
are precisely the subsets of the form $X_N = \{(g_1,g_2)\in G \times G: g_1N = g_2N\}$, where $N$ is a normal subgroup of $G$.  In other words, the prediagonal subsets are the graphs of the congruence relations modulo normal subgroups.
\par
3) In the category $\Rng$, given a ring $R$, the prediagonal subsets of $R \times R$ 
are precisely  the subsets of the form  $X_I = \{(r_1,r_2)\in R \times R: r_1+I = r_2+I\}$, where $I$ is a two-sided ideal of $R$. In other words, the prediagonal subsets are the graphs
of the congruence relations modulo two-sided ideals.
\par
4) In the category $R$-$\Mod$, given a $R$-module $M$, the prediagonal subsets of $M \times M$ are precisely the subsets of the form $X_N = \{(m_1,m_2)\in M \times M: m_1+N = m_2+N\}$, where $N$ is a submodule of $M$. In other words, the prediagonal subsets of $M \times M$ are the graphs of congruence relations modulo submodules of $M$.
Observe that every subdiagonal subset of $M \times M$ is a translate of a submodule of 
$M \times M$.
\par
5) In the full subcategory of $\Top$ whose objects are the Hausdorff topological spaces, given a Hausdorff topological space $A$, every square-algebraic (resp.~prediagonal,  resp.~codiagonal)   subset of  $A \times A$ is closed (resp.~closed, resp.~open)   in $A \times A$ for the product topology. 
 \par
6) In $\CHT$, given a compact Hausdorff topological space $A$, every subdiagonal subset of $A \times A$ is closed in $A \times A$.
\par
7)  In the category $K$-Aal of affine algebraic sets over a field $K$, every
square-algebraic (resp.~prediagonal,  resp. codiagonal)   subset of the square $A \times A$ of an object $A$ is closed (resp.~closed, resp.~open)   in the Zariski topology
on $A \times A$ (beware that the Zariski topology on $A \times A$ is not, in general, the product of the Zariski topology on $A$ with itself).
If $K$ is  algebraically closed,
it follows from Chevalley's theorem that every subdiagonal subset of $A \times A$ is constructible with respect to the Zariski topology on $A \times A$.
\end{examples}

The following definition is introduced by Gromov in \cite[Subsection 4.F]{gromov-esav}.

 \begin{definition}
Let $\CC$ be a concrete category satisfying (CFP). 
One says that $\CC$ has 
the \emph{subdiagonal intersection property}, briefly (SDIP), provided that the following holds: 
for every $\CC$-object $A$, any $\CC$-square-algebraic subset $X \subset A \times A$, any $\CC$-codiagonal subset $Y \subset A \times A$, and any   non-increasing sequence $(Z_n)_{n \in \N}$ of $\CC$-subdiagonal subsets of $A \times A$ such that $X \cap Y \cap Z_n \neq \varnothing$ for all $n \in \N$, one has
 $$
 \bigcap_{n \in \N} X \cap Y \cap Z_n \neq \varnothing.
 $$
\end{definition}

\begin{examples}
1) The category $\CC = \Set$ does not satisfy (SDIP). 
To see this, take for example $A = \N$ and consider the maps $f_n \colon A \to A$ defined by
$$
f_n(k) = \begin{cases}
k & \text{ if } k \leq n-1\\
0 & \mbox{ if } k \geq n
\end{cases}
$$
for all $n, k \in \N$. 
Then the sets
$$
Z_n := \{(a_1,a_2) \in A \times A: f_n(a_1) = f_n(a_2)\}
$$
form a non-increasing sequence of $\CC$-prediagonal (and therefore $\CC$-subdiagonal) subsets of $A \times A$. 
 Take $X = A \times A$ and $Y = A \times A \setminus \Delta_A$.
 Then $X \subset A \times A$ is  $\CC$-square-algebraic and $Y \subset A \times A$ is $\CC$-codiagonal in $A \times A$. 
  We have $X \cap Y \cap Z_n \not= \varnothing$ for all $n \in \N$ since $(n,n + 1)\in X \cap Y \cap Z_n$. On the other hand, we clearly have $\bigcap_{n \in \N} X \cap Y \cap Z_n = \varnothing$.
 This shows that $\Set$ does not satisfy (SDIP).
 \par
2) The category $\Setf$ satisfies (SDIP)
since any non-increasing sequence of finite sets eventually stabilizes.
\par
3) Let $\CC = \Rng$, $\Grp$, $\Z$-$\Mod$, or $\Z$-$\Modfg$.
Take $A = \Z$, $X = A \times A$, $Y = A \times A \setminus \Delta_A$, and, for   $n \in \N$,
$$
Z_n = \{ (a_1,a_2) \in A \times A : a_1 \equiv a_2 \mod 2^n \}.
$$
Clearly $X$ is $\CC$-square-algebraic, $Y$ is $\CC$-codiagonal, and $(Z_n)_{n \in \N}$ is a non-increasing sequence of $\CC$-prediagonal (and hence $\CC$-subdiagonal) subsets of $A \times A$. We have $X \cap Y \cap Z_n \not= \varnothing$ for all $n \in \N$ but
$\bigcap_{n \in \N} X \cap Y \cap Z_n = \varnothing$. This shows that $\CC$ does not satisfy (SDIP).
\par
4) Given an arbitrary ring $R$, the category $ R$-$\ModArt$  satisfies (SDIP).  Indeed
we have seen that every subalgebraic subset is the translate of some submodule and that, in an Artinian module,
every non-increasing sequence consisting of translates of submodules eventually stabilizes.
\par
5) Let $R$ be a left-Artinian ring.
Then the  category $R$-$\Modfg$ satisfies (SDIP) since it is a subcategory of $R$-$\ModArt$.
\par
5) Given a field $K$,  the category $K$-$\Vectfd$  satisfies (SDIP)
since $K$-$\Vectfd = K$-$\Modfg$ and every field is a left-Artinian ring.
\par
6) The category $\CHT$ of compact Hausdorff topological spaces does not satisfy (SDIP).
Indeed, let $A = [0,1]$ denote the unit segment and consider, for every $n \in \N$  the continuous map $f_n \colon [0,1] \to [0,1]$  defined by
$$
f_n(x) = \begin{cases}
x & \text{ if } x \leq \frac{n}{n+1}\\
\frac{n}{n+1}  & \text{ if } \frac{n}{n+1} \leq x  \\
 \end{cases}
$$
for all $x \in A$.
Take $X = A \times A$, $Y = A \times A \setminus \Delta_A$, and
$$
Z_n = \{(a_1,a_2) \in A \times A: f_n(a_1) = f_n(a_2)\}.
$$
Then $X$ is square-algebraic, $Y$ is codiagonal, and $(Z_n)_{n \in \N}$ is a non-increasing sequence of prediagonal (and therefore subdiagonal) subsets of $A \times A$.
We clearly have $X \cap Y \cap Z_n \not= \varnothing$ for all $n \in \N$ but $\bigcap_{n \in \N} X \cap Y \cap Z_n = \varnothing$. This shows that $\CHT$ does not satisfy (SDIP). 
 \par
7) A variant of the previous argument may be used to prove that even $\Man$ does not satisfy (SDIP).
Indeed, consider the circle $\sph^1 := \{z \in \C : \vert z \vert = 1 \}$ and,
 for each $n \in \N$, the continuous map $f_n \colon \sph^1 \to \sph^1$ defined by
 $$
f_n(z) = \begin{cases}
 \mathrm{e}^{\mathrm{i}\frac{n + 2}{n + 1}\theta }  & \text{ if }  z = \mathrm{e}^{  \mathrm{i}\theta}  \text{  with  }0 \leq \theta \leq \frac{n + 1}{n+2}2\pi\\
1  & \text{ otherwise. }   \\
 \end{cases}
$$
 Take $X = A \times A$, $Y = A \times A \setminus \Delta_A$, and
$$
Z_n = \{(a_1,a_2) \in A \times A: f_n(a_1) = f_n(a_2)\}.
$$
Then $X$ is square-algebraic, $Y$ is codiagonal, and $(Z_n)_{n \in \N}$ is a non-increasing sequence of prediagonal (and therefore subdiagonal) subsets of $A \times A$ 
in the category $\Man$.
We have that  $X \cap Y \cap Z_n \not= \varnothing$ for all $n \in \N$, since $(1,\mathrm{e}^{  \mathrm{i}\frac{n + 1}{n + 2}\pi }) \in Z_n$.  
On the other hand, we clearly have  $\bigcap_{n \in \N} X \cap Y \cap Z_n = \varnothing$. 
This shows  that $\Man$ does not satisfy (SDIP).
\par 
8)  Let $K$ be an uncountable algebraically closed field. 
Then the category $K$-$\Aal$   satisfies (SDIP) (cf. \cite[4.F']{gromov-esav}). 
This follows from the fact that, when $K$ is an uncountable algebraically closed field, every non-increasing sequence  of nonempty constructible subsets of an affine algebraic set over $K$  has a nonempty intersection (see \cite[Proposition 4.4]{algebraic}).
Indeed, if $X$ (resp.~$Y$, resp.~$Z_n$) is a $\CC$-square-algebraic (resp~prediagonal, resp.~subdiagonal) subset of $A \times A$, then, as observed above, $X$ (resp.~$Y$, resp.~$Z_n$) is closed (resp.~open, resp.~constructible) in $A \times A$ and hence $X \cap Y \cap Z_n$ is constructible since any finite intersection of constructible subsets is itself constructible. 
\end{examples}

\subsection{Reversibility of $\CC$-cellular automata}

We shall use the following auxiliary result.
 
\begin{lemma}
\label{l:proj}
Let $\CC$ be a concrete category satisfying (CFP) and (SDIP).
Let $(X_n,f_{nm})$ be a projective sequence of nonempty  sets.
Suppose that there is a projective sequence $(A_n,F_{n m})$, consisting of $\CC$-objects $A_n$ and $\CC$-morphisms $F_{n m} \colon A_m \to A_n$ for all $n,m \in \N$ such that $m \geq n$, satisfying the following conditions:
\begin{enumerate}[(PSD-1)]
\item
for all $n \in \N$, one has $X_n = Y_n \cap Z_n$, where $Y_n \subset A_n \times A_n$ is $\CC$-codiagonal and $Z_n \subset A_n \times A_n$ is $\CC$-prediagonal;
 \item
 for all $n,m \in \N$ with $m \geq n$, 
setting $S_{n m} = F_{n m} \times F_{n m}$, one has
$S_{n m}(X_m) \subset X_n$ and $f_{n m}$ is the restriction 
of $S_{n m}$ to $X_m$;
\item
 for all $n,m \in \N$ with $m \geq n$, one has
$S_{n m} (Z_m) \subset Z_n$ and 
$f_{nm}(X_m) = Y_n \cap S_{n m} (Z_m)$.  
\end{enumerate}  
Then $\varprojlim X_n \not= \varnothing$. 
\end{lemma}

\begin{proof}
  Let $(X_n',f_{nm}')$ denote the universal projective sequence
associated with the projective sequence $(X_n,f_{nm})$.
As   $S_{n m} (Z_m) $, $m = n, n+1, \ldots$, is a non-increasing sequence of 
$\CC$-subdiagonal subsets of $A_n \times A_n$ 
such that $Y_n \cap S_{n m} (Z_m)  = f_{n m}(X_m) \neq \varnothing$ for all $m \geq n$, we deduce from (SDIP) that
\begin{equation}
\label{eq:proj-1-satisfied}
X_n' = \bigcap_{m \geq n} f_{n m}(X_m) = \bigcap_{m \geq n} Y_n \cap S_{n m} (Z_m)  \neq \varnothing.
\end{equation}
 Let now $m,n \in \N$ with $m \geq n$ and suppose that $x_n' \in X_n'$.
Then we have    $f_{n m}^{-1}(x_n') \cap f_{m k}(X_k) \neq \varnothing$ for all $k \geq m$ 
  by Remark \ref{rem:non-empty-decreasing}.
By applying again (SDIP), we get
\begin{equation}
\label{eq:proj-2-satisfied}
\bigcap_{k \geq m} f_{n m}^{-1}(x_n') \cap f_{m k}(X_k) =  \bigcap_{k \geq m} F_{n m}^{-1}(x_n') \cap Y_m \cap S_{m k} (Z_k) \not= \varnothing
\end{equation}
(observe that $F_{n m}^{-1}(x_n')$ is $\CC$-square-algebraic).
From \eqref{eq:proj-1-satisfied} and \eqref{eq:proj-2-satisfied}, it follows that 
conditions  (IP-1) and (IP-2) in Corollary \ref{cor:IP} are satisfied,
so that we conclude from this corollary that $\varprojlim X_n \not= \varnothing$. 
\end{proof}

\begin{theorem}
\label{t:reversible}
Let $\CC$ be a concrete category satisfying (CFP+) and (SDIP), and let $G$ be an arbitrary group. 
Then every bijective $\CC$-cellular automaton $\tau \colon A^G \to B^G$ is reversible.
\end{theorem}

\begin{proof}
Let  $\tau \colon A^G \to B^G$ be a bijective $\CC$-cellular automaton.
We have to show that the inverse map $\tau^{-1} \colon B^G \to A^G$ is a cellular automaton. 
\par 
As in the proof of Theorem \ref{t:closed-image}, we suppose first that the group $G$ is countable. Let us show that the following local property is satisfied
by $\tau^{-1}$: 
\begin{itemize}
\item[($\ast$)]
there exists a finite subset $N \subset G$ such that, 
for any $y \in B^G$, the element $\tau^{-1}(y)(1_G)$ only depends on
the restriction of $y$ to $N$.
\end{itemize}
This will show that $\tau$ is reversible.
Indeed, if ($\ast$) holds for some finite subset $N \subset G$, then there exists a (unique) map  $\nu \colon B^{N} \to A$ such that
$$
\tau^{-1}(y)(1_G) = \nu(y\vert_N)  
$$
for all $y \in B^G$.
Now, the $G$-equivariance of $\tau$ implies the $G$-equivariance of its inverse map $\tau^{-1}$.
Consequently, we get
$$
\tau^{-1}(y)(g) = g^{-1}\tau^{-1}(y)(1_G) = \tau^{-1}(g^{-1}y)(1_G) =  \nu((g^{-1}y)\vert_N)
$$
for all $y \in B^G$ and $g \in G$. This implies that $\tau^{-1}$ is the cellular automaton with memory set $N$ and local defining map $\nu$.
\par
Let us assume by contradiction that condition  ($\ast$) is not satisfied. 
Let $M$ be a memory set for $\tau$ such that $1_G \in M$.
Since $G$ is countable, we can find a sequence $(E_n)_{n \in \N}$ of finite subsets of $G$ such that $G = \bigcup_{n \in \N} E_n$,
$M \subset E_0$, and $E_n \subset E_{n + 1}$ for all $n \in \N$. 
Consider, for each $n \in \N$, the finite subset $F_n \subset G$ defined by  
$F_n = \{g \in G: gM \subset E_n\}$. Note that $G = \bigcup_{n \in \N} F_n$,
$1_G \in F_0$,  and $F_n \subset F_{n + 1}$
for all $n \in \N$.
\par
Since ($\ast$) is not satisfied, we can find, for each $n \in \N$, two configurations 
$y_n', y_n'' \in B^G$ such that 
\begin{equation}
\label{e:property-y-n}
y_n'\vert_{F_n} = y_n''\vert_{F_n}
\quad \text{and} \quad 
\tau^{-1}(y_n')(1_G) \neq \tau^{-1}(y_n'')(1_G).
\end{equation}
Recall from the proof of Theorem \ref{t:closed-image}, 
that $\tau$ induces, for each $n \in \N$,  a $\CC$-morphism $\tau_n \colon A^{E_n} \to B^{F_n}$ given by $\tau_n(u) = (\tau(x))\vert_{F_n}$ for every $u \in A^{E_n}$, where $x \in A^G$ is any configuration extending $u$. 
\par
Consider now, for each $n \in \N$, the subset $X_n \subset A^{E_n} \times A^{E_n}$ 
consisting of all pairs $(u_n',u_n'')  \in A^{E_n} \times A^{E_n}$ such that 
$\tau_n(u_n') = \tau_n(u_n'')$ and $u_n'(1_G) \neq u_n''(1_G)$. 
We have  $X_n = Y_n \cap Z_n$, where
$$
Y_n := \{(u_n',u_n'') \in A^{E_n} \times A^{E_n} : u_n'(1_G) \not= u_n''(1_G) \}
$$
and
$$
Z_n := \{(u_n',u_n'') \in A^{E_n} \times A^{E_n} : \tau_n(u_n') = \tau_n(u_n'') \}.
$$
Note that $Y_n$ (resp.~$Z_n$) is a $\CC$-codiagonal (resp.~$\CC$-prediagonal) subset
of $A^{E_n} \times A^{E_n}$.   
 Note also that $X_n$ is not empty since 
$$
((\tau^{-1}(y_n'))\vert_{E_n},(\tau^{-1}(y_n''))\vert_{E_n}) \in X_n
$$
by \eqref{e:property-y-n}.
\par
For $m \geq n$, the restriction 
map $\rho_{n m} \colon A^{E_m} \to A^{E_n}$
gives us a $\CC$-square-morphism
$$
S_{n m}= \rho_{n m} \times \rho_{n m} \colon A^{E_m} \times A^{E_m} \to A^{E_n} \times A^{E_n}
$$
which induces by restriction a map $f_{n m} \colon X_m \to X_n$.
\par
We clearly have $S_{n m}(Z_m) \subset Z_n$ and $S_{n m}(X_m) =  Y_n \cap S_{n m}(Z_m)$
for all $n,m \in \N$ such that $m \geq n$.
 Since $\CC$ satisfies (SDIP), it follows from Lemma \ref{l:proj},
that $\varprojlim X_n \not= \varnothing$.
\par
Choose an element $(p_n)_{n \in \N} \in \varprojlim X_n$. 
Thus $p_n = (u_n',u_n'') \in A^{E_n} \times A^{E_n}$  and $u'_{n + 1}$ (resp. $u''_{n + 1}$) coincides with $u_n'$ (resp. $u_n''$) on $E_n$  for all $n \in \N$. 
As $G = \cup_{n \in \N} E_n$, we deduce that there exists a (unique) configuration 
$x' \in A^G$ (resp. $x'' \in A^G$) such that $x'\vert_{E_n} = u_n'$ 
(resp. $x''\vert_{E_n} = u_n''$) for all $n \in \N$. 
Moreover, we have 
$$
(\tau(x'))\vert_{F_n}= \tau_n(u_n') = \tau_n(u_n'') =  (\tau(x''))\vert_{F_n} 
$$
for all $n \in \N$.   
As $G = \cup_{n \in \N} F_n$, this shows that $\tau(x') = \tau(x'')$.
On the other hand, we have
$x'(1_G) = u_0'(1_G) \not= u_0''(1_G) = x''(1_G)$ and hence $x' \not= x''$.
This contradicts the injectivity of $\tau$ and 
therefore completes the proof that $\tau$ is reversible in the case when $G$ is countable.
\par
We now drop the countability assumption on $G$ and prove the theorem in its full  generality. Choose a memory set $M \subset G$ for $\tau$ and denote by $H$ the subgroup of $G$ generated by  $M$.
Observe that $H$ is countable since $M$ is finite.
By Theorem \ref{t:induction}, the restriction cellular automaton $\tau_H \colon A^H \to A^H$ is a bijective $\CC$-cellular automaton. It then follows from the first part of the proof that $\tau_H$ is reversible.
This implies that $\tau$ is reversible as well by Theorem \ref{t:induction}.(iii).
\end{proof}

\begin{corollary}
\label{cor:reversible}
If $G$ is an arbitrary group, 
all bijective $\CC$-cellular automata are reversible   when $\CC$ is one of the following concrete
categories:
\begin{itemize}
\item 
$\Setf$, the category of finite sets;
\item 
$K$-$\Vectfd$, the category of finite-dimensional vector spaces over an arbitrary field $K$;
\item 
$R$-$\ModArt$, the category of left Artinian modules over an arbitrary ring $R$;
\item 
$R$-$\Modfg$, the category of finitely generated left modules over an arbitrary left-Artinian 
ring $R$;
\item 
$K$-$\Aal$, the category of affine algebraic sets over an arbitrary uncountable algebraically closed field $K$.
 \end{itemize}
\end{corollary}

In the following examples we show that there exist bijective cellular automata
which are not reversible. They are modeled after \cite{periodic} (see also \cite[Example 1.10.3]{book}).

\begin{examples}
\label{ex:p-adic}
1) Let $p$ be a prime number and $A = \Z_p = \varprojlim \Z/p^n\Z$ the ring of $p$-adic integers. 
Recall that $A$ is a compact Hausdorff topological ring  for the topology associated with the 
$p$-adic metric.
We can regard $A$ as a $\CC$-object for $\CC = \Set, \Grp, \Z$-$\Mod$, and  $\CHT$.
Consider now the cellular automaton $\tau \colon A^\Z \to A^\Z$ defined by
\[
\tau(x)(n) = x(n) - px(n + 1)
\]
for all $x \in A^\Z$ and $n \in \Z$. It has memory set $M = \{0,1\} \subset \Z$ and
associated local defining map $\mu_M \colon A^M   \to A$ given by
$\mu_M(y) = y(0) - py(1)$ for all $y \in A^M$. It follows that $\tau$ is a $\CC$-cellular
automaton for $\CC$ any of the concrete categories mentioned above.
Note that $\tau$ is bijective. Indeed the inverse map $\tau^{-1} \colon A^\Z \to A^\Z$
is given by
\[
\tau^{-1}(x)(n) = \sum_{k = 0}^\infty  p^kx(n + k)
\]
for all $x \in A^\Z$ and $n \in \Z$. However, $\tau^{-1}$ is not a cellular automaton.
Indeed, let $F$ be a finite subset of $\Z$ and choose an integer $m \geq 0$ such that
$F \subset (-\infty,m]$. Consider the configurations $y,z \in A^\Z$ defined by $y(n) = 0$
if $n \leq m$ and $y(n) = 1$ if $n \geq m+1$, and $z(n)=0$ for all $n \in \N$, respectively.
Then $y$ and $z$ coincide on $F$. However, we have
\[
\tau^{-1}(y)(0) = \sum_{k = 0}^\infty  p^ky(k) = \sum_{k = m+1}^\infty  p^k
\]
and
\[
\tau^{-1}(z)(0) = \sum_{k = 0}^\infty  p^kz(k) = 0.
\]
It follows that there is no finite subset $F \subset \Z$ such that $\tau^{-1}(x)(0)$
only depends on the restriction of $x \in A^\Z$ to $F$. Thus, there is no
finite subset $F \subset \Z$ which may serve as a memory set for $\tau^{-1}$.
\par\label{ex:power-series}
2) Let $R$ be a ring and let $A = R[[t]]$ denote the ring of all formal power series in
one indeterminate $t$ with coefficients in $R$. Note that $A$ has a natural structure 
of a left $R$-module.
Then the cellular automaton $\tau \colon A^\Z \to A^\Z$ defined by
\[
\tau(x)(n) = x(n) - tx(n + 1)
\]
for all $x \in A^\Z$ and $n \in \Z$ is a bijective $\CC$-cellular automaton for
$\CC = R$-$\Mod$.
However, this cellular automaton  is not reversible unless $R$ is a zero ring. The proof is analogous to the one given  in the preceding example
(see \cite[Example 1.10.3]{book} in the case when $R$ is a field).
\end{examples}

\begin{remark}
In Examples \ref{ex:p-adic},  we can replace the
group $\Z$ by any non-periodic group $G$. Indeed, if $G$ is non-periodic and
$H \subset G$ is an infinite cyclic subgroup (thus $H \cong \Z$), 
the induced cellular automaton $\tau^G \colon A^G \to A^G$
is a bijective $\CC$-cellular automaton for $\CC =  \Set, \Grp, \Z$-$\Mod$, and $\CHT$
(resp. $R$-$\Mod$ with $R$ a nonzero ring) by Theorem \ref{t:induction} and Proposition \ref{p:C-induction}, 
which is not reversible.
\end{remark}

\end{document}